%% file: accepted_colpatterns.tex
\documentclass[11pt, dvipsnames]{amsart}

\usepackage{amsthm,amsmath,amssymb,bm,mathtools} 
\usepackage{graphicx}
\usepackage{xcolor}
\usepackage{verbatim}
\usepackage{float}
\usepackage{bbm}
\usepackage{enumerate}
\usepackage{hyperref}
\usepackage{pstricks-add}
\usepackage{pst-hsb}

\usepackage[english,noresetcount,norelsize,ruled]{algorithm2e}

\usepackage{array}   % for \newcolumntype macro
\newcolumntype{C}{>{$}c<{$}}
\newcolumntype{L}{>{$}l<{$}}

\theoremstyle{plain}
\newtheorem{theo}{Theorem}[section]
\newtheorem{prop}[theo]{Proposition}
\newtheorem{lemma}[theo]{Lemma}

\theoremstyle{definition}

\newtheorem{problem}[theo]{Problem}

\newcommand{\mc}[1]{\mathcal{#1}}
\newcommand{\mb}[1]{\mathbb{#1}}
\newcommand{\mbf}[1]{\mathbf{#1}}

\newcommand{\sm}{\setminus}

\newcommand{\eps}{\varepsilon}

\newcommand{\aA}{\alpha}
\newcommand{\bB}{\beta}
\newcommand{\gG}{\gamma}
\newcommand{\dD}{\delta}

\title{
Universality for transversal Hamilton cycles}
\author{Candida Bowtell}
\author{Patrick Morris}
\author{Yanitsa Pehova}
\author{Katherine Staden}
\thanks{CB: School of Mathematics, University of Birmingham, Edgbaston, Birmingham, United Kingdom, \href{mailto:c.bowtell@bham.ac.uk}{\tt c.bowtell@bham.ac.uk}, supported by ERC Starting Grant 947978 and Leverhulme Trust Early Career Fellowship ECF--2023--393. \\
\indent PM: Departament de Matem\`atiques, Universitat Polit\`ecnica de Catalunya (UPC),  Barcelona, Spain, \href{mailto:pmorrismaths@gmail.com}{\tt pmorrismaths@gmail.com}, supported by
		the Deutsche Forschungsgemeinschaft (DFG, German Research
		Foundation) Walter Benjamin program - project number
		504502205.\\
\indent YP: Department of Mathematics, London School of Economics, United Kingdom, \href{mailto:y.pehova@lse.ac.uk}{\tt y.pehova@lse.ac.uk}, supported by the Engineering and Physical Sciences Research Council, UK Research and Innovation [EP/V038168/1]. \\
\indent KS: School of Mathematics and Statistics, The Open University, Walton Hall, Milton Keynes, United Kingdom, \href{mailto:katherine.staden@open.ac.uk}{\tt katherine.staden@open.ac.uk}, supported by the Engineering and Physical Sciences Research Council, UK Research and Innovation [EP/V025953/1].
}

\begin{document}

\date{}

\begin{abstract} 
Let $\mathbf{G}=\{G_1, \ldots, G_m\}$ be a graph collection on a common vertex set $V$ of size $n$ such that $\delta(G_i) \geq (1+o(1))n/2$ for every $i \in [m]$. We show that $\mathbf{G}$ contains every Hamilton cycle pattern. That is, for every map $\chi: [n] \to [m]$ there is a Hamilton cycle whose $i$-th edge lies in $G_{\chi(i)}$. 
\end{abstract}

\maketitle

\section{Introduction}

The problem of determining whether a given graph contains a Hamilton cycle -- a cycle containing every vertex -- is a central problem in graph theory. It appears on Karp's famous list of NP-complete problems~\cite{karp}, and as such, much research has focused on finding sufficient conditions to guarantee the existence of a Hamilton cycle.
The best-known result of this type is Dirac's theorem from 1952~\cite{dirac}, which states that, in a graph $G$ on $n\ge 3$ vertices, minimum degree $\dD(G) \geq n/2$ suffices. Graphs satisfying these conditions are often called \emph{Dirac graphs}. There is a large literature on extensions of this theorem and other sufficient conditions for Hamiltonicity in various graph classes.

\subsection{Transversal embedding}

A recent fruitful research direction has been to extend classical results about graph embedding to graph collections. The general problem was posed by Joos and Kim~\cite{jooskim}, as follows. Given a collection $\mathbf{G}=\{G_1,\ldots,G_m\}$ of not necessarily distinct graphs on a common vertex set $V$, what conditions on these graphs imply that there is a copy of a given graph $J$ with at most $m$ edges, containing at most one edge from any $G_i$? 
 Such a copy of $J$ is called a \emph{transversal} or \emph{rainbow} copy of $J$, the latter coming from thinking of each $G_i$ as having a distinct colour.
In particular, if each $G_i$ satisfies the same condition $\mc{C}$ guaranteed to ensure that $G_i$ contains a copy of $J$ -- we say that \emph{$\mathbf{G}$ satisfies $\mc{C}$} -- does this imply that the collection $\mathbf{G}$ contains a transversal copy of $J$?
If so, we informally say that containing $J$ is \emph{colour-blind} with respect to $\mc{C}$. 
Proving that if every graph in $\mathbf{G}$ satisfies some property $\mc{C}$, then $\mathbf{G}$ contains a transversal copy of $J$ generalises the original embedding problem of proving that if a graph $G$ satisfies $\mc{C}$, then $G$ contains a copy of $J$, as we could have all graphs in $\mbf G$ being identical to $G$.
As usual in extremal graph theory, we seek \emph{best possible} conditions $\mc{C}$, meaning that for any weakening $\mc{C}'$ of $\mc{C}$, there is a graph collection satisfying $\mc{C}'$ which does not contain $J$.

There have been several Tur\'an-type transversal results, where $J$ is a fixed graph, $n$ is sufficiently large, and $\mc{C}$ is a lower bound on the number of edges in an $n$-vertex graph.
The classical theorem of Tur\'an from 1941~\cite{turan}, extending Mantel's theorem for triangles from 1907~\cite{mantel}, states for any $n$-vertex graph $G$ that $e(G) > e(T_{r-1}(n))$ is sufficient to guarantee that $G$ contains a copy of the $r$-clique $K_r$, where $T_{r-1}(n)$ is the $n$-vertex $(r-1)$-partite graph with parts as equal in size as possible (and of course, $T_{r-1}(n)$ itself witnesses that this is best possible).
Aharoni, DeVos, de la Maza, Montejano and \v{S}\'amal~\cite{Aharoni} showed that the transversal generalisation of Mantel's theorem is \emph{not} colour-blind. If $G_1,G_2,G_3$ are three graphs on the same $n$-vertex set, and $e(G_i)> tn^2$, where 
$t=\frac{26-2\sqrt{7}}{81} \approx 0.2557 > 1/4$, then there is a transversal copy of $K_3$, and this is best possible. The extension to larger cliques is still wide open. Recently, Babi\'nski, Grzesik and Prorok~\cite{bgp} obtained a version for oriented triangles in directed graphs.

There has also been a lot of progress in the setting where $J$ is spanning, and $\mc{C}$ is a minimum degree condition. Joos and Kim~\cite{jooskim} proved that the transversal generalisation of Dirac's theorem \emph{is} colour-blind, sharpening an approximate version due to Cheng, Wang and Zhao~\cite{Cheng1}. 
Given the difficulty of obtaining exact results in graph embedding problems, many results have focused on determining %\emph{approximate colour-blindness}, introduced in~\cite{mmp}, where we seek 
the minimum $d$ such that for all $\eps>0$ and $n$ sufficiently large, $\dD(G_i) \geq (d+\eps)n$ for all $G_i$ in a collection $\mathbf{G}$ of $n$-vertex graphs implies that $\mathbf{G}$ contains a transversal copy of $J$.
Note that $d$ is at least the value required to guarantee a copy of $J$ in a single graph, since all the graphs in the collection could be identical.
In this paper we consider uniform minimum degree conditions of this type, and we denote by $\delta(\mbf G)$ the smallest of all $\delta(G_i)$.

Transversal results have been obtained for matchings~\cite{jooskim}, $F$-factors for a fixed graph $F$~\cite{mmp}, spanning trees~\cite{mmp}, powers of Hamilton cycles~\cite{ghmps}, graphs of small bandwidth~\cite{CIKL}, Hamilton paths and cycles in tournaments~\cite{chakraborti2023hamilton} as well as
Hamilton cycles and other spanning structures in hypergraphs~\cite{Cheng2,ghmps}.
Each of these asymptotically generalise classical theorems in extremal graph theory, and in most cases the problem is ``approximately colour-blind'', but not always, for example for some $F$-factors (see \cite[Proposition 6.4]{mmp}). 

\subsection{Pattern embedding}

In this paper, we study a further generalisation of the transversal problem, where one seeks a copy of $J$ with a given colouring. %, which we call a \emph{colour pattern} of $J$.

\begin{problem}\label{prob:1}
    Given a graph collection $\mathbf{G}=\{G_1,\ldots,G_m\}$ on a common vertex set and a graph $J$, what conditions on $\mathbf{G}$ guarantee that for every colouring $\chi: E(J) \to \{1,\ldots,m\}$ of the edges of $J$, there is a copy of $J$ such that the image of $e$ is an edge of $G_{\chi(e)}$ for all edges $e$ of $J$?
\end{problem}

We call $\chi$ a \emph{colour pattern} or \emph{colouring} of $J$. When every graph in the collection is identical, or, equivalently, if $m=1$, we again recover the usual graph embedding problem. When $\chi$ is a bijection (so $m=e(J)$), we recover a stronger version of the transversal embedding problem.
Indeed, the transversal embedding problem asks for \emph{any} rainbow copy, whereas Problem~\ref{prob:1} asks for \emph{any given} colouring, which may be rainbow or may contain repeated colours, for example a Hamilton cycle with alternating red and blue edges or a Hamilton cycle made up of half-length red and blue paths. As we will show shortly, it suffices to solve the problem  for bijections $\chi$; that is, when the colour pattern we seek is rainbow, and every colour is used.

When $J$ is a matching and $\chi$ is a bijection, any rainbow copy of $J$ is a specific rainbow copy, since we can permute the edges.
Montgomery, M\"uyesser and Pehova~\cite{mmp} studied the instance of this problem where $J$ is an $F$-factor and each copy of $F$ is a distinct colour.
As far as we are aware, there are no other results towards Problem~\ref{prob:1}.

Graph collections and colour patterns appear in another guise in the area of \emph{temporal graphs}. A temporal graph is a graph which changes over time, so it can be represented as a graph collection $\mathbf{G}=\{G_1,\ldots,G_m\}$ in which $G_i$ is the copy of the graph at time $i$.
For example, a path with the colour pattern mapping the $i$-th edge in the natural order to colour $i$ corresponds to, in the temporal graph, following a path using at step $i$ the graph available at time $i$. This is known as a \emph{temporal journey}.
This perspective is of interest in the theoretical computer science community, and in particular determining how the complexity of graph algorithms changes in the temporal setting. We refer the interested reader to the survey of Michail~\cite{michail} for an introduction to this area.

\subsection{Main result}

The main result of this paper is an asymptotically best possible answer to Problem~\ref{prob:1} when $J$ is a Hamilton cycle and the condition is on minimum degree, showing that minimum degree $(1/2+o(1))n$ guarantees a Hamilton cycle with any given colour pattern. This answers a question of M\"{u}yesser.
\begin{theo}\label{thm:general} For every $\alpha>0$ there exists $n_0$ such that for every $n \geq n_0$ the following holds.
Let $m \in \mathbb{N}$ and let $\mbf{G}=\{G_1, \ldots, G_m\}$ be a graph collection on a common vertex set $V=[n]$, such that $\delta(\mathbf{G})\geq (1/2+\alpha)n$. Then, for every $\chi:[n] \to [m]$, there is a Hamilton cycle $C=e_1\ldots e_n$ such that $e_i \in G_{\chi(i)}$ for all $i \in [n]$.
\end{theo}

As promised, Theorem~\ref{thm:general} can be easily deduced from the following special case where $m=n$ and $\chi$ is the identity colouring. We use round brackets to emphasise the implicit ordering of $\{G_1,\ldots,G_n\}$ which encodes this canonical colouring. 

\begin{theo}\label{thm:main} For every $\alpha>0$ there exists $n_0$ such that for every $n \geq n_0$ the following holds.
Let $\mbf{G}=(G_1, \ldots, G_n)$ be an ordered graph collection on a common vertex set $V=[n]$, such that $\delta(\mathbf{G})\geq (1/2+\alpha)n$. Then $\mathbf{G}$ contains a Hamilton cycle $C=e_1\ldots e_n$ such that $e_i\in G_i$ for all $i\in[n]$.
\end{theo}

Theorem~\ref{thm:general} is obtained from Theorem~\ref{thm:main} by replicating graphs whose colour must appear more than once. Indeed, 
let $\chi:[n] \to [m]$. We may assume that $m \leq n$ by removing graphs from the collection, since the image of $\chi$ has size at most $n$.
For each $j\in[m]$ let $A_j:=\chi^{-1}(j)\subseteq [n]$ be the preimage of $j$ under the colouring. These preimages form a partition $[n]=A_1 \cup \ldots \cup A_m$.
For each $j\in[m]$ and $i\in A_j$, let $H_i:=G_{j}$.
Theorem~\ref{thm:main} applied to the ordered collection $(H_1,\ldots,H_n)$ yields a Hamilton cycle as required by Theorem~\ref{thm:general}.

\subsection{Tightness }
\label{sec:counterex}
 As with the transversal version of Dirac's theorem proven by Joos and Kim \cite{jooskim}, a minimum degree condition of $\delta(\mathbf{G})\geq n/2$ is certainly necessary for Theorem \ref{thm:main} (and Theorem \ref{thm:general}) as otherwise one can simply take $\bf{G}$ to be copies of some non-Hamiltonian graph. 
 This shows that our results are asymptotically tight. 
 However, unlike the case of finding only one transversal Hamilton cycle, it turns out that in the context of guaranteeing \emph{every} colour pattern of a Hamilton cycle, a condition of $\delta(\mathbf{G})\geq n/2$ is \emph{not} sufficient. This follows from the following construction due to Gupta, Hamann, Parczyk and Sgueglia (personal communication).  Let $n$ be even and let ${\bf{G}}=\{G_1,G_2\}$ with $G_1$ consisting of the union of two disjoint copies of $K_{\frac{n}{2}}$ with a perfect matching between them, and $G_2$ a copy of $K_{\frac{n}{2}, \frac{n}{2}}$, where the vertices are partitioned in the same way as for $G_1$. We have that $\delta(\mathbf{G})\geq n/2$ but $\mathbf{G}$ has no  Hamilton cycle with two consecutive edges from $G_1$ and the remaining $n-2$ edges in $G_2$. To see this, note that starting the cycle with two edges from $G_1$ either means using three vertices in one copy of $K_{\frac{n}{2}}$ or starting with an edge in one copy of $K_{\frac{n}{2}}$ followed by an edge to the other copy. In both cases, to complete to the desired colour pattern we may take only crossing edges, but completing to a Hamilton cycle requires an edge inside one of the parts.

 The authors do not know of any counterexamples to suggest that $\delta(\mathbf{G})\geq n/2+1$ is not sufficient to guarantee all Hamilton cycle patterns. Moreover, this example only yields a lower bound of at least $ n/2+1$ for even $n$. Determining a tight non-asymptotic minimum degree condition remains an intriguing open problem; see our concluding remarks for more on this. 

 \subsection{Related counting results}

Counting different Hamilton cycles in Dirac graphs has also received a lot of attention, and recently this has been  extended to transversal Hamilton cycles. Bradshaw, Halasz, and Stacho~\cite{bhs} showed that given a collection of $n$ Dirac graphs on the same vertex set $[n]$, there are at least $(cn/e)^{cn}$ different transversal Hamilton cycles for some constant $c \geq 1/68$. From Theorem~\ref{thm:main} we immediately recover the stronger lower bound, that there exist at least $n!\sim \sqrt{2\pi n}(n/e)^n$ different transversal Hamilton cycles in a collection of $n$ asymptotically Dirac graphs. Very recently, however, Anastos and Chakraborti~\cite{ac} independently improved this to show that there exists a constant $C>0$ such that there are $(Cn)^{2n}$ different transversal Hamilton cycles in a collection of $n$ Dirac graphs, which is tight up to the choice of $C$.
Just as our result on patterns gives a weaker (but non-trivial) lower bound on counting transversal Hamilton cycles, by pigeonhole their result yields a weaker but non-trivial lower bound on the proportion of all patterns of Hamilton cycles that can be found in any collection of $n$ Dirac graphs.

\subsection{Proof strategy}
Our proof follows an absorption strategy and thus splits the theorem into the tasks of finding an appropriate absorbing structure and an almost spanning substructure of the desired Hamilton cycle. 
Although many of the previous approaches \cite{CIKL,chakraborti2023hamilton,Cheng2,chwwy,Cheng1,ghmps,mmp} to transversal problems also appeal to the same general strategy, we remark that none of them (and indeed no other previous proof schemes \cite{staden_rainbow,jooskim}) apply to pattern embedding problems as they all crucially use the flexibility given by being able to produce \emph{any} rainbow structure rather than one specified colour pattern. 
This is discussed in more detail in Section \ref{sec:proof overview}.

Our  absorbing structure will output a subpath of our Hamilton cycle and provide flexibility for which vertices will appear in this path. Crucially, we require that it outputs the \emph{same segment} of the desired edge-coloured Hamilton cycle, no matter which set of vertices are left to absorb. In order to achieve this, we construct an absorbing gadget which can absorb any one of some constant number of vertices into a small path with a fixed colouring. Gadgets are then assigned to vertex sets using the robustly matchable bipartite graph introduced by Montgomery \cite{mont} as a template. 

For the almost spanning substructure we split the desired edge-coloured path into many small subpaths and use a random process to find these subpaths one at a time. Throughout this process, we use Dirac's theorem for bipartite graphs (Lemma~\ref{lem:hall}) as a black box, in order to guarantee matchings which we glue together  to obtain many copies of our desired subpath. 

We refer the reader to Section \ref{sec:proof overview} for a more detailed overview of the proof.

\subsection{Organisation} In Section \ref{sec:notation} we fix some notation and provide some basic tools and results that will be used in our proofs. In Section \ref{sec:proof overview} we then outline our proof and give two key lemmas, Lemma \ref{lemma:abspath} and Lemma \ref{lemma:pathcoll}. We then use these lemmas to prove our main theorem in Section \ref{sec:mainproof}. Section \ref{sec:absorption} is then devoted to proving Lemma \ref{lemma:abspath} and Section \ref{sec:path cover} to Lemma \ref{lemma:pathcoll}. Finally in Section \ref{sec:conclude} we give some remarks on further avenues of research. 

\section{Notation and preliminaries}
\label{sec:notation}
We use standard graph-theoretic notation. Given a graph $G$, $v \in V(G)$ and $U \subseteq V(G)$, 
we write $N_G(v,U) := N_G(v) \cap U$ and $ d_G(v,U) := |N_G(v,U)|$.
We often think of $G$ as a set of edges, so use $G$ as shorthand for its edge-set $E(G)$.
We denote by $P=e_1e_2\ldots e_s$ the path of length $s$, which has $s+1$ vertices and $s$ edges,
and by $C=e_1e_2\ldots e_s$ the cycle of length $s$, which has $s$ vertices and $s$ edges.
We denote the first vertex of $P$ by $v_P^-$ and the last vertex of $P$ by $v_P^+$.
An \emph{$x,y$-path} is a path whose degree-1 vertices are $x$ and $y$. 

Throughout the paper we often omit floors and ceilings in the interest of readability. In any such case the result can be made precise without non-trivial changes to the proof. We use standard notation for the hierarchy of constants. That is, the statement ``Let $0<a\ll b\ldots$" is used to mean that there is  a non-decreasing function $f$ such that all following statements concerning these constants hold for any choice of constants $0<a\le f(b)$. Longer hierarchies are defined analogously, with constants being chosen from right to left.  

We now give notation specific to graph collections and transversal subgraphs. Given an ordered graph collection $\mathbf{G}=(G_1, \ldots, G_n)$ on a common vertex set $V$, we define $\delta(\mathbf{G}):= \min_{G \in \mathbf{G}} \delta(G)$.  For vertex subsets $A,B\subset V$, we define $\mathbf{G}[A,B]$ to be the collection of graphs $(G_1[A,B],\ldots,G_n[A,B])$  induced between sets $A$ and $B$ and $\mathbf{G}[A]$ is defined similarly.  We say that graph $G_i$ {\it has colour $i$} and for any edge-coloured graph $J$ on vertex set $V$ with edges coloured in colours from $[n]$, we say that {\it $J$ is a subgraph of $\mathbf{G}$} if for every $i \in [n]$ and every edge $e$ in $J$ of colour $i$ we have $e \in G_i$.
Given a graph $H$, a map $\chi: E(H) \rightarrow [n]$ is a \emph{colour pattern} of $H$.
When $H$ is a path of length $s \leq n$, we often use the index set $[s]$ of its edges to define a colour pattern. 
Given a map $\chi: [s] \rightarrow [n]$ for some $s \leq n$, let $P=e_1e_2\ldots e_s$ be a path of length $s$ with $V(P) \subseteq V$. We say that $P$ is {\it exactly $\chi$-coloured} if for every $i \in [s]$, $e_i \in G_{\chi(i)}$. 

In the rest of this section we give a few standard graph theory tools we will use for the rest of the paper. The first is the well-known concentration inequality of Chernoff. The following lemma holds for any binomial random variable $X$. As discussed in \cite[Section 21.5]{frieze2015introduction}, concentration results for the binomial distribution transfer to the hypergeometric distribution.

\begin{lemma}[Chernoff bounds {\cite[Corollary 21.7]{frieze2015introduction}}]    \label{lem:chernoff}
    Let $\eps>0$ and let $X$ be a binomially or hypergeometrically distributed random variable with mean $\mu$. Then
    \[\mb P(|X-\mu|\ge \eps \mu)\le 2e^{-\eps^2\mu/3}.\]

\end{lemma}

To embed long paths in our graph collection, we will use a simple consequence of Hall's matching theorem which we state without proof.%\ks{Refer to this also as Dirac's theorem in bipartite graphs, since this is how we refer to it elsewhere?}
\begin{lemma}\label{lem:hall}
    Let $G=(A,B)$ be a bipartite graph with $|A|=|B|=n$ and $\delta(G)\geq n/2$. Then $G$ contains a perfect matching.
\end{lemma}

This is also sometimes referred to as Dirac's theorem for bipartite graphs. Our next proposition shows how we can partition the vertex set of a graph and maintain  degrees to both parts. 
\begin{prop}\label{prop:Z}
Let $0 < 1/n \ll \alpha, \beta$. 
Let $\mathbf{G}=(G_1, \ldots, G_n)$ be a graph collection on a common vertex set $V=[n]$ with $\delta(\mbf G) \geq (1/2+\alpha)n$. Then there exists $Z \subseteq V$ such that $|Z|= \beta n$ and for every $i \in [n]$ and $v \in V$, 
%\begin{linenomath}
\begin{equation*}  d_{G_i}(v,Z) \geq (1/2+\alpha/2)|Z|
\quad\text{and}\quad
 d_{G_i}(v, V\setminus Z) \geq (1/2+\alpha/2)|V\setminus Z|.\end{equation*} %\end{linenomath}
\end{prop}
\begin{proof}
Let $Z$ be a uniformly random chosen subset of $V$ of size $\beta n$. Let $Z^i_v$ be the degree of $v$ into $Z$ in $G_i$ and $X^i_v$ be the degree of $v$ into $V \setminus Z$ in $G_i$. Note that $\mathbb{E}(Z^i_v)=\beta d_{G_i}(v) \pm 1 \geq (1/2+\alpha)\beta n-1$ and $\mathbb{E}(X^i_v)=(1-\beta) d_{G_i}(v) \pm 1 \geq (1/2+\alpha)(1-\beta) n -1$ for every $v \in V$ and $i \in [n]$.\footnote{$\pm 1$ coming from whether $v \in Z$ or $v \in V\setminus Z$ in each case.} Furthermore, both $Z_v^i$ and $X_v^i$ are hypergeometric. 
Thus, via Lemma~\ref{lem:chernoff} (Chernoff bounds), we have with probability less than $\exp(-\sqrt{n})$ that $Z^i_v \leq (1/2+\alpha/2)\beta n$ and that $X^i_v \leq (1/2+\alpha/2)(1-\beta) n$. Then taking a union bound over $Z^i_v$ and $X^i_v$ for every $v \in V$ and every $i \in [n]$ yields the existence of $Z$.
\end{proof}

Throughout the proof of our main result, we will use the following proposition multiple times to join  some exactly $\chi_1$-coloured path $P_1$ and another exactly $\chi_2$-coloured path $P_2$ via a cherry (a path of length two) with middle vertex $z$ and edges in colours $c_1$ and $c_2$. (With slight abuse of notation, this produces an exactly  $\chi_1 c_1c_2\chi_2$-coloured path $P_1zP_2$.) We refer to such connective steps as \emph{cherry connections}.
\begin{prop}\label{prop:cherries}
Let $0 < 1/n \ll \alpha$.
Let $G_1,G_2$ be two graphs on a common vertex set $V=[n]$ and  let $Z\subseteq V$ such that for every $v\in V$, $ d_{G_1}(v,Z)\ge (1/2+\alpha/2)|Z|$ and $ d_{G_2}(v, Z)\ge (1/2+\alpha/2)|Z|$. Then for every pair of vertices $x, y \in V$ and every $U \subseteq V$ such that $|U\cap Z| <\alpha|Z|$, there exists $z \in Z \setminus (U \cup \{x,y\})$ such that $xzy$ is a cherry with $xz \in G_1$ and $yz \in G_2$. 
\end{prop}

\begin{proof}
    By inclusion-exclusion, $|N_{G_1}(x,Z)\cap N_{G_2}(y,Z)|\ge \alpha |Z|$ so $|(N_{G_1}(x,Z)\cap N_{G_2}(y,Z))\sm U|>0$, i.e.~there is a vertex $z$ in the common neighbourhood avoiding $U$. Then $xzy$ is the required cherry.
\end{proof}

\section{Main lemmas and proof overview}
\label{sec:proof overview}

In this section we discuss the main ideas of our proof and state our two main lemmas, namely Lemmas \ref{lemma:abspath} and \ref{lemma:pathcoll},  which together imply our main result Theorem \ref{thm:main}.

First though, we briefly discuss  the previous approaches to transversal problems for spanning structures and why they cannot be adapted to our setting. Indeed, the approach of \cite{mmp} which is also adopted in \cite{CIKL,chakraborti2023hamilton,ghmps} uses a colour absorption strategy finding fixed subgraphs that have flexibility in the colours that can be used on them. Similarly the absorbers used in \cite{Cheng2,chwwy,Cheng1} provide flexibility for which vertices and which colours are used to contribute to some fixed part of the spanning structure but this flexibility results in a loss of control over which  colour patterns appear on  the subgraphs of the spanning structure. The approach of Joos and Kim \cite{jooskim} follows a `rotation-extension' approach as in the proof of Dirac's theorem, which destroys any fixed colour pattern of a subpath in the process of extending the path or closing it to a cycle. Finally, the regularity-based method of \cite{staden_rainbow} views the problem of finding a transversal structure in a graph as an embedding problem in $3$-uniform hypergraphs by adding the set of colours as new vertices and a $3$-uniform edge for each original edge of the collection with its corresponding colour. When embedding structures in this hypergraph, one loses the `labelling' of the colours and hence the ability to find specific colour patterns. 

The one exception in previous approaches to transversal problems being able to be strengthened to the case of finding all colour patterns is  the proof of Montgomery, M\"uyesser and Pehova~\cite{mmp} in the  case of clique factors. In short, the idea is that the configuration of $\binom{r}{2}$ colours in any copy of $K_r$ in a coloured $K_r$-factor can be viewed as one ``compound colour'', reducing the colour pattern problem to a classical transversal problem on the compound colours.
However, this approach does not apply when the sought spanning structure has $o(n)$ components. 

In our case, it is clear that a new strategy is necessary. Similarly to some of the previous approaches \cite{CIKL,chakraborti2023hamilton,Cheng2,chwwy,Cheng1,ghmps,mmp}, we use the method of absorption but our absorbing structure needs to provide flexibility \emph{only for vertices}. That is, the colours that will be used on the part of the spanning structure given by the absorption will be fixed beforehand and our absorption strategy will produce a fixed subpath of our edge-coloured Hamilton cycle. In detail, we will obtain the following.

\begin{lemma}[Absorbing path] \label{lemma:abspath}
Let $0<1/n \ll \gamma \ll \beta \ll \alpha \leq 1$. Then there exists  $a\le 500\beta n$ such that the following holds. Let   $t:=a+\beta n+1$ and let $\mathbf{G}=(G_1, \ldots, G_t)$  be a graph collection on a common vertex set $V=[n]$ with $\delta(\mathbf{G})\geq (1/2+\alpha)n$. Let $Z \subseteq V$ be a vertex set of size $(\beta +\gamma )n+2$ and $z_1,z_2\in Z$ be a pair of vertices. Then there exists an \emph{absorbing set} $A \subseteq V \setminus Z$ of size $a$ such that for every $Z' \subseteq Z\setminus \{z_1, z_2\}$ of size $\beta n$, there is a $z_1,z_2$-path $P_{Z'}=e_1 \ldots e_t$ whose internal vertices cover $A \cup Z'$ exactly and $e_i\in G_i$ for all $i\in[t]$.
\end{lemma}

The proof of Lemma \ref{lemma:abspath} is given in Section \ref{sec:absorption}. We build $A$ as a union of many constant-size absorbing gadgets as shown in Figure~\ref{fig:absorbing-gadget} (see page \pageref{fig:absorbing-gadget}). Each gadget will have vertex set $L\cup R$ with two distinct special vertices $x,y\in R$ and will have the key property that for \emph{any} $v\in L$  there is a rainbow path (with a fixed rainbow colour pattern) from $x$ to $y$ using the vertices of $R$ and the vertex $v$. This gadget thus provides flexibility: it will always contribute the same part of the rainbow Hamilton cycle that we desire, but we can choose which vertex in $L$ contributes to this. Moreover, the absorbing gadget is carefully defined so that it has degeneracy two and so can be found (in a robust way) in our graph collection (Lemma \ref{lemma:absgag}).
Lemma \ref{lemma:abspath} is then achieved by building the appropriate absorbing gadgets according to some \emph{template} given by a robustly matchable bipartite graph. 
This approach was pioneered by Montgomery \cite{mont} and allows one to transfer the local flexibility given by the absorbing gadgets to a more global flexibility given by Lemma \ref{lemma:abspath}, which allows us to choose any $\beta n$-sized subset $Z'$ of a larger set $Z$  to contribute to  our fixed rainbow path. 

As is usual with the absorption approach, after putting aside the absorbing structure given by Lemma \ref{lemma:abspath} on vertex set $A\cup Z$, we then need to find an \emph{almost spanning} structure on the vertices $V\setminus (A\cup Z)$. We will then use some of the vertices of $Z$ to extend the almost spanning structure and connect the sections of our Hamilton cycle, all the time using the colours dictated by our colour pattern. 
We will do this in a way that leaves a fixed number of vertices of $Z$ uncovered so that the key property of Lemma \ref{lemma:abspath} will complete the desired Hamilton cycle with the parts that we have already found. 

In order to extend and connect through the vertex set $Z$  we choose $Z$ randomly so that all vertices have high degree (in all colours) to $Z$. Doing this allows us to make many connections through $Z$ (in a disjoint fashion) using Proposition~\ref{prop:cherries} and so instead of our almost spanning structure being a connected component of our Hamilton cycle, we can afford to chop it up into linearly many paths, each of which has some large constant $K$ many vertices. This is useful  for us, as it allows us to find these segments of the Hamilton cycle independently without worrying about connecting them until we use the vertices in $Z$. We will find these subpaths by taking a random partition of $V\setminus (A\cup Z)$ and applying our second main lemma which we now state.

\begin{lemma}[Path collection] \label{lemma:pathcoll} 
Let $0<1/n \ll \eps,1/K  \ll \alpha \leq 1$ and let $\mathbf{G} = (G_1, \ldots, G_n)$ be a collection of graphs on a common vertex set $V=[n]$ which are all $K$-partite on the balanced vertex partition $V_1, \ldots, V_K$, such that $\delta(\mathbf{G}[V_j, V_{j+1}]) \geq (1/2 + \alpha)n/K$ for all $j \in [K-1]$. Let $s\le(1-\eps)n/K $ and for each $i \in [s]$ let $\chi_i: [K-1] \rightarrow [n]$. Then there exist vertex-disjoint paths $P_1, \ldots, P_s$ each on $K$ vertices such that $P_i$ is exactly $\chi_i$-coloured for every $i\in[s]$.
\end{lemma}

The proof of Lemma \ref{lemma:pathcoll} is given in Section \ref{sec:path cover}. In order to find the paths, we run a simple random process which removes randomly chosen paths one at a time.  In more detail, we will use that, given that the degree of each vertex (in each colour) to a neighbouring part is at least half the size of the parts at a given time, Lemma \ref{lem:hall} guarantees the existence of a matching in each colour between parts. Choosing the coloured matchings appropriately, we thus get a factor of rainbow paths, each with the rainbow colour pattern we desire. We choose one of these paths uniformly at random and move to the next step of the process. In order for this to work we need to show that with  high probability the minimum degree condition between the vertex sets is maintained. Intuitively this should be the case as within one part $V_i$ the set we have removed so far in the algorithm is a uniformly random set and so we expect the degree of a vertex to the remainder to be proportional to its original degree in $V_i$. We show this property is maintained with high probability using a `R\"odl nibble'  type analysis \cite{rodl1985packing}, treating several rounds of the algorithm together in small bites. 

\section{Proof of main result} \label{sec:mainproof}

Here we put together Lemma \ref{lemma:abspath} (Absorbing path) and Lemma \ref{lemma:pathcoll} (Path collection) to prove Theorem~\ref{thm:main} which, as discussed in the introduction, also implies the more general Theorem \ref{thm:general}.

\begin{proof}[Proof of Theorem \ref{thm:main}] Given $\alpha>0$ we let $0<1/n\ll \eps,1/K \ll \gamma \ll \beta \ll \alpha$. 

We start by finding $Z\subseteq V$ of size $(\beta + \gamma)n+2$, as per Proposition \ref{prop:Z}. We have that for every $G_i$ and $v \in V$, $ d_{G_i}(v,Z) \geq (1/2+\alpha/2)|Z|$ and $ d_{G_i}(v,V\setminus Z) \geq (1/2+\alpha/2)|V\setminus Z|$. Arbitrarily choose $\{z_1, z_2\} \subseteq Z$. By Lemma \ref{lemma:abspath} there exists a set $A \subseteq V \setminus Z$ of size $a \leq 500\beta n$ such that for every subset $Z' \subseteq Z\setminus\{z_1,z_2\}$ of size $\beta n$, there exists a $z_1,z_2$-path $P_{Z'}=e_1e_2\ldots e_t$ of length $t:=a+\beta n +1$ such that $e_i \in G_i$ for every $i \in [t]$ and the internal vertices of $P_{Z'}$ cover $A\cup Z'$ exactly.

In order to find the required Hamilton cycle $C$, it remains to find a $z_2, z_1$-path $P_L=e_{t+1}e_{t+2}\ldots e_n$ of length $n-t$ in $V\setminus A$  which covers $V\setminus (A \cup Z)$, leaves exactly $\beta n$ vertices in $Z$, and has  $e_i \in G_i$ for every $i \in \{t+1, \ldots, n\}$. Let $V':=V\setminus (A \cup Z)$, $n':=|V'|$, and $\mathbf{G'}:=(G'_1, \ldots, G'_n)$ where $G'_i:=G_i[V']$ for every $i \in [n]$. Note that by definition we have $n'=n-t-\gamma n-1$. Since $\beta, \gamma \ll \alpha$, we have that $\delta(\mathbf{G'}) \geq (1/2+\alpha/2)n$. Consider a random partition of $V'$ into $K$ sets $(V_1, \ldots, V_K)$ of size $\lfloor{n'/K}\rfloor$ and one set, $V_0$, of size $n'-K\lfloor{n'/K}\rfloor \leq K$. By Lemma~\ref{lem:chernoff} (Chernoff bounds) 
and taking a union bound it follows that there exists such a partition with the property that for every $j \in [K-1]$, $\delta(\mathbf{G'}[V_j, V_{j+1}]) \geq (1/2 + \alpha/4)\lfloor{n'/K}\rfloor$.

Fixing such a partition, we have that, by Lemma \ref{lemma:pathcoll}, for $s :=\lfloor (1-\eps)\lfloor n'/K\rfloor \rfloor$ and any maps $\chi_1, \ldots, \chi_s$ with every $\chi_i:[K-1] \rightarrow [n]$, there exist vertex-disjoint paths $P_1, \ldots, P_s$ such that $P_i$ is exactly $\chi_i$-coloured. For each $j\in [K-1]$ and $i\in [s]$ we take %$\chi_i(j)=t+i(K+2)-K+j$.
$\chi_i(j):=t+i(K+1)-K+1+j$.

It remains to make the following connections, with each new vertex or path added using only additional vertices in $(V' \cup Z)\setminus\{z_1, z_2\}$ such that each path is internally disjoint from any previous paths:
\begin{enumerate}[(i)]
\item connect $z_2$ to $v_{P_1}^-$ via a cherry in colours $t+1$ and $t+2$ in order, 
\item for each $i \in [s-1]$, connect $v_{P_i}^+$ to $v_{P_{i+1}}^-$ via a cherry in colours $\chi_i(K-1)+1$ and $\chi_i(K-1)+2=\chi_{i+1}(1)-1$ in order, and
\item find a path starting at $v_{P_s}^+$ and ending at $z_1$ that covers any remaining vertices in $V'$ (including those in $V_0$), of length exactly $n-\chi_s(K-1)$ in colours $\chi_s(K-1)+1$ up to $n$ in order.
\end{enumerate}
Note that  there are at most $n'/K \ll \gamma n$ cherry connections as in~(i) and~(ii) to make. Thus by Proposition \ref{prop:cherries} applied with the set of vertices already used in a path segment playing the role of $U$, these connections can be made greedily through $Z$. Furthermore, there are $c:= n'-sK\leq K\left(\eps \lfloor n'/K\rfloor+1\right) \ll \gamma n$ vertices left to cover in $V'$. 
Giving them an arbitrary ordering $v_1, \ldots, v_c$ we greedily make cherry connections via unused vertices in $Z\sm \{z_1,z_2\}$ in successive colours starting from $\chi_s(K-1)+1$ between $v_{P_s}^+$ and $v_1$, and between $v_i$ and $v_{i+1}$ for every $i \in [c-1]$ as per Proposition \ref{prop:cherries}. 
Note that up to this point we have a $z_2,v_c$-path whose internal vertices are in $V' \cup Z$ and cover $V'$ and far fewer than $\gamma n$ vertices of $Z\sm \{z_1,z_2\}$, meaning that in particular more than $\beta n$ vertices remain unused in $Z$. To complete step~(iii) it remains to find a path from $v_c$ to $z_1$ through unused vertices in $Z\sm \{z_1,z_2\}$ leaving exactly $\beta n$ vertices of $Z\sm \{z_1,z_2\}$ uncovered, and using all colours up to $n$ exactly in order. 

To do this, starting from $v_c$, we will first pick unused neighbours in $Z \setminus \{z_1, z_2\}$ in successive colours greedily until exactly $\beta n+1$ vertices of $Z\sm \{z_1,z_2\}$ remain uncovered by our path. This is possible via our application of Proposition~\ref{prop:Z}. 
%Recall that $|Z|=(\beta+\gamma)n+2$. Thus, combined with the fact that to complete steps~(i)--(ii), we used much fewer than $\gamma n$ vertices of $Z$, we have the required space left in $Z$ to complete step~(iii) leaving exactly $\beta n$ vertices of $Z \setminus \{z_1, z_2\}$ uncovered, as required. In particular, by our application of Proposition~\ref{prop:Z} we are able to complete step~(iii) greedily. % as there are at most $\gamma n +1 \leq |Z|/2$ neighbours for any verte In particular we have $\Theta(\gamma n) \ll \beta n$ have already been covered, leaving at least $(1-o(1))\gamma n\gg 1$ for the length of this greedy process. On the other hand we want to leave exactly $\beta n+1$ vertices of $Z\sm\{z_1,z_2\}$ uncovered, so at most $\gamma n - 1<|Z|/2$ vertices of $Z\sm \{z_1,z_2\}$ will be covered at any given time.
%By our application of  Proposition \ref{prop:Z}, we have that $ d_{G_i}(v,Z)\geq (1/2+\alpha/2)|Z|$ for all $i\in [n]$ and all $v\in V$, so the greedy process will produce the required path. 
Finally, to complete step~(iii), by Proposition \ref{prop:cherries}, we can complete the path to $z_1$ with a cherry in the next two colours in the ordering, again via an unused vertex in $Z\sm\{z_1,z_2\}$, leaving a set $Z'\subseteq Z\sm\{z_1,z_2\}$ of exactly $\beta n$ vertices uncovered. 

By construction, the $z_2,z_1$-path produced in steps (i)--(iii) uses consecutive colours starting at $t+1$. Its internal vertices are exactly $V'$ together with a $\gamma n$-sized subset of $Z\sm\{z_1,z_2\}$. From this, its length is exactly $|V'|+\gamma n+1=n-t$, meaning that the last edge has colour $n$.

Finally, taking the path $P_{Z'}$ guaranteed by the absorbing property of $A$ given by Lemma~\ref{lemma:abspath}  completes the desired Hamilton cycle.
\end{proof}

\section{Absorption} \label{sec:absorption}

We use the increasingly popular absorption strategy due to Montgomery~\cite{mont}, sometimes referred to as \emph{distributive absorption} or the \emph{absorption template method}, which was introduced as an absorbing structure for spanning trees. This strategy relies on an auxiliary template known as a \emph{robustly matchable bipartite graph}. 
This allows us to create few disjoint absorbing gadgets relative to all possible combinations of sets which might need to be absorbed.  

\begin{lemma}\label{lemma:rmbg}
Let $0 < 1/m \ll \beta$.
Then there exists a bipartite graph $B$ with vertex classes $X_B$ and $Y_B \cup Z_B$ such that $|X_B|=3m$, $|Y_B|=2m$, $|Z_B|=m+ \beta m$, $2 \leq \dD(B) \leq \Delta(B) \leq 40$, and the following holds. For every subset $Z_B'\subseteq Z_B$ with $|Z_B'|=m$, there exists a perfect matching in $B[X_B, Y_B \cup Z_B']$.
\end{lemma}

%(For technical reasons we need $\delta(B)\ge 2$ which can be obtained by adjoining to $B$ an arbitrary matching covering any degree-1 vertices in $Y_B\cup Z_B$.) \pmo{Is it clear that there are at most $3m$ deg 1 vxs in $Y_B\cup Z_B$?}
Here, we use a version of the robustly matchable bipartite graph which requires a weak minimum degree condition for the construction of our absorbers to make sense. This lemma follows from the statement and proof of \cite[Lemma 2.8]{mont2}, where the graph given in fact has minimum degree at least 20 and maximum degree at most 40, though the statement does not mention the minimum degree as it was not needed for the main result in \cite{mont2}. %noting that the construction of the graph in fact given by Nenadov and Pehova in \cite{np}, adding a weak minimum degree condition to their statement.

\subsection{Absorbing gadgets}

Let $\ell \in \mathbb{N}$ and let $\mbf G = (G_1,\ldots,G_{4\ell+2})$ be a graph collection on a common vertex set $[n]$. Given a set $L$ of $\ell+1$ vertices we say that an edge-coloured graph $F$ with vertex set $L\cup R$ is an \emph{$L$-absorbing gadget} with \emph{endpoints} $x,y\in R$ if for every $v\in L$ there exists an $x,y$-path $P_v=e_1e_2\ldots e_{4\ell+2}\subseteq F$ satisfying
\begin{enumerate}
\item $V(P_v) = R\cup \{v\}$, and
\item $e_i \in G_i$ for every $i \in [4\ell +2]$.
\end{enumerate}

We now give an explicit construction of such an absorbing gadget. We define the edge-coloured graph $F_\ell$ as follows. 
Let $V(F_\ell) =A \cup B \cup C$ be a disjoint union of vertices such that $A=\{a_1, \ldots, a_{\ell+1}\}$, $B=\{b_0, \ldots, b_{3\ell+1}\}$ and $C=\{c_1, \ldots, c_{\ell}\}$. Then let $E(F_{\ell})=\{a_{i+1}b_{3i}, a_{i+1}b_{3i+1}: i \in \{0, 1, \ldots, \ell\}\} \cup \{c_{i+1}b_{3i}, c_{i+1}b_{3i+1}, c_{i+1}b_{3i+3}, c_{i+1}b_{3i+4}: i \in \{0, 1, \ldots, \ell-1\} \}\cup \{b_{3i+1}b_{3i+2}, b_{3i+2}b_{3i+3}: i \in \{0, \ldots, \ell-1\}\}$. 
Now colour edges $a_1b_0$ and $b_0c_1$ with colour $1$, and edges $a_{\ell+1}b_{3\ell}$ and $b_{3\ell}c_{\ell}$ with colour $4\ell+1$. 
We also colour edges $a_1b_1$, $b_1c_1$ with colour $2$ and $a_{\ell+1}b_{3\ell+1}$, $b_{3\ell+1}c_{\ell}$ with colour $4\ell+2$. 
For $i \in [\ell-1]$, we give $a_{i+1}b_{3i}$, $b_{3i}c_{i}$, and $b_{3i}c_{i+1}$, colour $4i+1$,
and $a_{i+1}b_{3i+1}$, $b_{3i+1}c_i$, and $b_{3i+1}c_{i+1}$ colour $4i+2$. Finally, for $i \in \{0, \ldots, \ell-1\}$ and $j \in \{1,2\}$, we colour edges $b_{3i+j}b_{3i+j+1}$ in colour $4i+j+2$. Note that $F_\ell$ is not rainbow (see Figure~\ref{fig:absorbing-gadget} for an illustration). We now show that $F_\ell$ is an $A$-absorbing gadget with $R=B\cup C$.

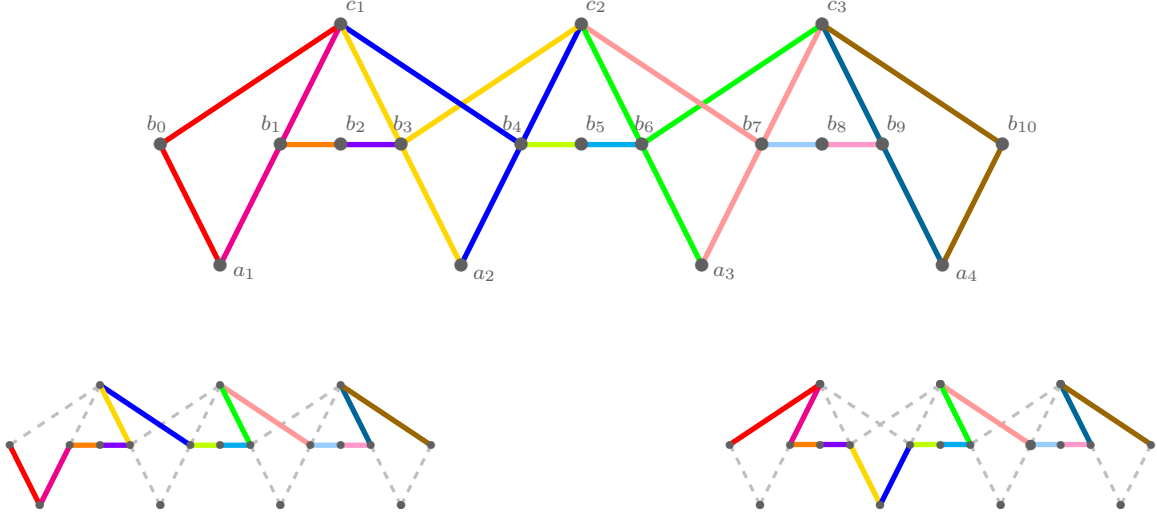
\begin{figure}
    \centering\input{absorbing_gadget}
    \caption{The absorbing gadget $F_3$ and the way it absorbs $a_1\in A$ and $a_2\in A$. The other two vertices $a_3$ and $a_4$ are absorbed symmetrically.}
    \label{fig:absorbing-gadget}
\end{figure}

\begin{prop}\label{prop:absgag}
Let $\ell\in \mb N$, let $\mathbf{G}=(G_1, \ldots, G_{4\ell+2})$ be a graph collection on a common vertex set $V=[n]$, and suppose that $A\subseteq V$ is a set of $\ell+1$ vertices. Then any copy of $F_\ell$ in $\mbf G$ with vertex set $A\cup B\cup C$ is an $A$-absorbing gadget with endpoints $b_1,b_{3\ell+1}\in B$.
\end{prop}
\begin{proof}
To prove this proposition, it suffices to show that for every $a_i\in A$ we can find a $b_0,b_{3\ell+1}$-path $P_{a_i}=e_1e_2 \ldots e_{4\ell+2}$ in $F_\ell$ such that  
\begin{enumerate}
\item $V(P_{a_i})=B\cup C\cup \{a_i\}$, and
\item $e_i \in G_i$ for every $i \in [4\ell +2]$.
\end{enumerate}

For this, suppose first that $i=1$. Then $P_{a_1}$ consists of the segments $b_0a_1b_1b_2b_3$ and, for $i \in [\ell-1]$, $b_{3i}c_{i}b_{3i+1}b_{3i+2}b_{3i+3}$, and $b_{3\ell}c_{\ell}b_{3\ell+1}$. Similarly if $i=\ell+1$, then take $P_{a_{\ell+1}}$ consisting of the segments  $b_{3i}c_{i+1}b_{3i+1}b_{3i+2}b_{3i+3}$ for $i \in \{0,\ldots, \ell-1\}$, and $b_{3\ell}a_{\ell+1}b_{3\ell+1}$. Suppose now that $1< i \leq \ell$. Then $P_{a_i}$ consists of the following segments. For $1\leq j <i$, take $b_{3j-3}c_jb_{3j-2}b_{3j-1}b_{3j}$, for $j = i$ take $b_{3j-3}a_jb_{3j-2}b_{3j-1}b_{3j}$, for $i<j \leq \ell$ take $b_{3j-3}c_{j-1}b_{3j-2}b_{3j-1}b_{3j}$ and finally take $b_{3\ell}c_{\ell}b_{3\ell+1}$.
\end{proof}

We claim that we can greedily find sufficiently disjoint $L$-absorbing gadgets for a suitable collection of sets $L$, as required, each of which is a copy of $F_\ell$ with $A=L$.

\begin{lemma}[Robust existence of absorbing gadgets]\label{lemma:absgag}
Let $0<1/n \ll 1/\ell, \alpha \leq 1$ and let $\mathbf{G}=(G_1, \ldots, G_{4\ell+2})$ be a graph collection on a common vertex set $V=[n]$ with $\delta(\mathbf{G}) \geq (1/2+\alpha)n$. Moreover, let $U \subseteq V$ be such that $|U| \leq \alpha n/4$. Then for every $L \subseteq V$ of size $\ell+1$ there is an $L$-absorbing gadget in $\mbf G$ avoiding vertices in $U\setminus L$. 
\end{lemma}

The proof of this lemma will be a simple consequence of the next auxiliary result.
We say that a graph $G$ has {\it degeneracy $k$} if there exists an ordering of the vertices of $G$ such that every vertex has at most $k$ previous neighbours in the ordering. The next lemma essentially tells us that given a graph collection $\mathbf{G}=\{G_1, \ldots, G_m\}$ of graphs on a common vertex set $[n]$ satisfying $\delta(\mathbf{G})\geq (\frac{k-1}{k}+\aA)n $, for any small graph $J$ of degeneracy $k$  and $\chi: E(J) \rightarrow [m]$, there robustly exist exactly $\chi$-coloured copies of $J$ in $\mathbf{G}$. In fact the lemma says something slightly stronger, allowing us to extend a partial embedding of $J$ respecting $\chi$ under certain additional conditions.  
We will only apply the lemma in the case $Z=V$ and $k=2$.

\begin{lemma}\label{lemma:degenk}
Let $0<1/n \ll \aA,1/k < 1$ and 
let $\mathbf{G}=\{G_1,\ldots,G_m\}$ be a graph collection on a common vertex set $V=[n]$. Further, 
 let $Z \subseteq V(G)$
 %with $|Z| \geq \aA n$ 
 be such that $ d_{G_i}(v,Z) \geq \left(\frac{k-1}{k}+\aA\right)|Z|$ for all $i \in [m]$
and $v \in V$.
Let $J$ be a graph on at most $\aA|Z|/4$ vertices and let $\chi : E(J) \to [m]$ be an edge-colouring of $J$.
Let $I \subseteq V(J)$ be an independent set of $J$ and let $f: I \to V$ be an injection.
Suppose that there is an ordering of $V(J)$ with initial segment $I$ (in any order)
such that each vertex $v$ has at most $k$ previous neighbours.
Let $U \subseteq V \sm f(I)$ satisfy $|U \cap Z| \leq \aA|Z|/4$.

Then there is an exactly $\chi$-coloured copy of $J$ avoiding $U$
where $x$ is mapped to $f(x)$ for each $x \in I$, and every vertex in $V(J) \sm I$ is mapped into $Z$.
\end{lemma}

A special case is Proposition~\ref{prop:cherries}, when $J$ is a cherry and $I$ is the pair of endpoints. 

\begin{proof}[Proof of Lemma \ref{lemma:degenk}]
Let $|V(J)|=:h$ and $v_1, \ldots, v_h$ be an ordering of the vertices of $J$ with initial segment $I$ such that each vertex has at most $k$ previous neighbours. For each $i \in [h]$, let $V_i:=\{v_1, \ldots, v_i\}$.  We will use these sets to describe the set of previous neighbours of each vertex in the ordering of $V(J)$, in particular, $ d_J(v_{i+1},V_i)\le k$ for all $i\in [h]$ by our degeneracy assumption. Assume $I$ is an independent set in $J$, and that the injection $f:I \rightarrow V$ is given. 
We find an injective map $g:V(J) \rightarrow (f(I)\cup Z)\sm U$ such that $g|_I=f$ and 
%$\mbf G[g(V(J))]$ contains an exactly $\chi$-coloured copy of $J$. 
$g(u)g(v)\in G_{\chi(uv)}$ for all $uv\in E(J)$.
Suppose that $|I|=s$. Then we map the vertices $v_{s+1}, \ldots, v_h$ one-by-one, as follows. Suppose we have already mapped $v_{1}, \ldots, v_{i}$ for some $s \leq i < h$. Then we wish to map $v_{i+1}$. Now, for every vertex $v \in N_{J}(v_{i+1},V_i)$, we want $g(v_{i+1})\in Z\sm U$ to be such that $g(v)g(v_{i+1}) \in G_{\chi(vv_{i+1})}$. Looking at $N_{G_{\chi(vv_{i+1})}}(g(v),Z)$ for every $v \in N_{J}(v_{i+1},V_i)$, since $ d_{J}(v_{i+1},V_i) \leq k$ and $ d_{G_i}(v,Z) \geq (\frac{k-1}{k}+\alpha)|Z|$ for all $i \in [m]$, it follows that $|\bigcap_{v \in N_{J}(v_{i+1},V_i)} N_{G_{\chi(vv_{i+1})}}(g(v),Z)| \geq \alpha |Z|$. 
Thus, avoiding the at most $\alpha|Z|/4$ vertices in $U \cap Z$ and at most $\alpha|Z|/4$ previously chosen vertices $g(v_1),\ldots,g(v_i)$, there are at least $\alpha|Z|/2$ choices for $g(v_{i+1})\in Z\sm U$. 
Fixing any such choice for each of $v_{s+1}, \ldots, v_h$ defines $g$ as required. 
That is, such a map yields an exactly $\chi$-coloured copy of $J$ in $\mbf G[f(I)\cup Z]$ avoiding $U$, where $x$ is mapped to $f(x)$ for each $x \in I$. 
\end{proof}

\begin{proof}[Proof of Lemma \ref{lemma:absgag}]
Let $0<1/n \ll 1/\ell,\aA \leq 1$, and let $L \subseteq V$ with $|L|=\ell+1$. By Proposition \ref{prop:absgag} and Lemma \ref{lemma:degenk} applied with $Z=V$ and $k=2$, it suffices to show that $F_\ell$ has an ordering of the vertices such that $L$ forms the initial segment and every vertex has at most two previous neighbours. Then any copy of $F_\ell$ in $\mbf G$ (guaranteed by Lemma~\ref{lemma:degenk}) with initial segment $L$ will be an $L$-absorbing gadget. One such ordering of $V(F_\ell)$ is to take $a_1,\ldots,a_{\ell+1},b_0,b_1,s_1,\ldots,s_{\ell},b_2,b_5,\ldots,b_{3i-1},\ldots,b_{3\ell-1}$, where for each $i \in [\ell]$, $s_i=c_i,b_{3i},b_{3i+1}$.
\end{proof}

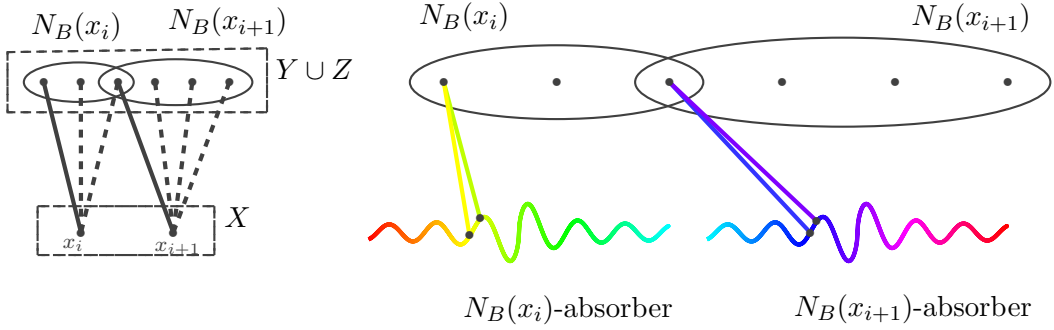
\begin{figure}
    \centering
    \input{absorption}
    \caption{Two consecutive vertices $x_i$ and $x_{i+1}$ in $B$, their neighbourhoods, and an illustration of their absorbing gadgets. Matching $x_i$ and $x_{i+1}$ via the two solid edges on the left corresponds to absorbing the two vertices shown on the right.}
    \label{fig:absorbing-path}
\end{figure}

\subsection{The absorbing path}

We now turn to proving Lemma \ref{lemma:abspath}. Our strategy is to use the bipartite graph $B$ given by  Lemma \ref{lemma:rmbg} as an auxiliary template which instructs our choice of where to place absorbing gadgets. In more detail, we will take such a bipartite graph $B$ and label the vertices so that   $Z_B\subseteq V(B)$ corresponds to the set $Z\setminus \{z_1,z_2\}\subset V=[n]$ given in Lemma \ref{lemma:abspath} and $Y_B$ corresponds to some other subset $Y\subset V\setminus Z$ which we can choose arbitrarily. Then for each $x\in X_B$ we will find some $L_x$-absorbing gadget in $\mathbf{G}$ where $L_x\subset Y\cup Z$ is the set of vertices corresponding to $N_B(x)\subset Y_B\cup Z_B$. These absorbing gadgets will be chosen to be vertex-disjoint outside of $Y\cup Z$. All the absorbing gadgets will then be connected via repeated applications of Proposition \ref{prop:cherries} to get one absorbing path whose vertex set $A\cup Z$ comprises all the vertices used in absorbing gadgets and in connections. In order to see that this concatenation of absorbing gadgets fulfills the requirements of Lemma \ref{lemma:abspath}, we will use the key property of the bipartite graph $B$ in Lemma \ref{lemma:rmbg} which guarantees that after deleting vertices from $Z\setminus \{z_1,z_2\}$ to obtain some $Z'$ (and taking $Z'_B\subset Z_B$ corresponding to $Z'$), there is a matching in $B$ between $Y_B\cup Z_B'$ and $X_B$. This in turn defines which vertex to include in the path output by the $L_x$-absorbing gadget, for each $x\in X$.  See Figure \ref{fig:absorbing-path} for an illustration of this. The union of these paths given by the individual gadgets will then be the required path covering $A\cup Z'$. The details follow.

\begin{proof}[Proof of Lemma \ref{lemma:abspath}]
Let $0<1/n \ll \gG \ll \bB \ll \aA \leq 1$, and let $s:=\beta n$ be an integer. 
Let $X_B,Y_B$ and $ Z_B$ be disjoint sets of vertices of size $3s$, $2s$ and $s + \gG  n$ respectively. 
By Lemma \ref{lemma:rmbg} there exists a bipartite graph $B$ with parts $X_B=\{x_1, \ldots, x_{3s}\}$ and $Y_B \cup  Z_B$ such that $2 \leq \dD(B) \leq \Delta(B) \leq 40$ and for every subset $Z_B' \subseteq  Z_B$ of size $s$, $B[X_B, Y_B \cup Z_B']$ has a perfect matching. 
Let $a := 4e(B)-\bB n + 1 \leq 500 \bB n$ and let $t = a+\beta n+1= 4e(B)+2$ be as given.

Let $\mbf G = (G_1,\ldots,G_t)$ be a graph collection on $V=[n]$, where $\dD(\mbf {G}) \geq (1/2+\aA)n$ and let $Z \subseteq V$ of size $s+\gG n+2$ with two vertices $z_1,z_2\in Z$ as given in the statement of the lemma.  We further fix some  $Y \subseteq V \sm Z$ of size $2s$.
We identify the  vertices of $Y_B\subset V(B)$ with the vertices of $Y\subset V$ and likewise the vertices of $Z_B\subset V(B)$ with the vertices in $Z\sm \{z_1,z_2\}\subset V$.

For each $i \in [3s]$, we let $L_i\subset Y\cup Z\setminus \{z_1,z_2\}$ be the set corresponding to $N_B(x_i)\subset Y_B\cup Z_B$ and let $\ell_i=d_B(x_i)-1$ (here we are using $\delta(B)\ge 2$ to ensure $\ell_i$ is a positive integer).  We will build an $L_i$-absorbing gadget in a subcollection of $\mbf{G}$ of size exactly $4\ell_i+2=4d_B(x_i)-2$ such that these subcollections are disjoint for distinct $i$. 
For each $i \in [3s]$, define colours
%\begin{linenomath} 
\begin{equation*}
\mu_i := 4d_B(x_1)+\ldots+4d_B(x_{i-1})+3 
\quad\text{and}\quad
\nu_i := \mu_i + 4d_B(x_i)-3 = 4d_B(x_1)+\ldots + 4d_B(x_i)
 \end{equation*} 
 %\end{linenomath}
and a graph collection $\mbf G^i := (G_{\mu_i},G_{\mu_i+1},\ldots,G_{\nu_i})$.
Note that $\mu_1=3=\mu_{i+1}-\nu_i$ for all $i \in [3s-1]$, and $\nu_{3s}=4e(B)$.
Informally, we are breaking the set of colours from $\mbf G$ into chunks of exactly the right size to build each required absorbing gadget leaving gaps of length exactly two so that we can subsequently connect these via cherry connections in precisely the required colours (using Proposition~\ref{prop:cherries}).

Set $U_1:=Y\cup Z$, $A_0:=Y$ and for each $i\in [3s]$ in turn, apply Lemma~\ref{lemma:absgag} to find an $L_i$-absorbing gadget $F^i$ in $\mbf G^i$ avoiding $U_i\setminus L_i$,  with endpoints labelled as $b^i_-,b^i_{+}$.
Set $U_{i+1}:=U_i\cup V(F^i)$ and $A_i:=(A_{i-1}\cup V(F^i))\sm Z$. 
Note that each such absorbing gadget uses exactly $4 d_B(x_i)-2\le 158$ vertices outside $Y \cup Z$ so at each step $i$ we have $|U_i|\le |U_{3s}|\le  3s\cdot 158+|Y\cup Z|\le 478\beta n \leq \alpha n/4$, as required by Lemma~\ref{lemma:absgag}.

Now, in order to complete the proof of the lemma, we iteratively apply Proposition~\ref{prop:cherries} (using that $\delta(\mbf G)\geq (1/2+\alpha)|V|$ here) to connect $z_1$ to $b_-^1$ with a cherry using colours $1$ and $2$ (that is, graphs $G_1$ and $G_2$) in this order, $b_{+}^{3s}$ to $z_2$ with a cherry using colours $\nu_{3s}+1=t-1$ and $\nu_{3s}+2=t$, and for each $i \in [3s -1]$, $b_{+}^i$ to $b_-^{i+1}$ with a cherry using colours $\nu_i+1$ and $\nu_i+2=\mu_{i+1}-1$, each time using vertices which have not already been used. Note that every time we wish to find such a cherry, there are at most $|U_{3s}|+3s+1< 500\beta n <\alpha n$ vertices to avoid. Add all internal vertices of the $3s+1$ cherries constructed in this step to $A_{3s}$ to obtain the set $A$. Note that $|A_{3s}|=4e(B)-6\beta n+|A_0|$ since each $A_i$ adds exactly $4 d_B(x_i)-2$ vertices to $A_{i-1}$.
We have $
|A| = |A_{3s}|+ (3s+1) = 4e(B)-\bB n+1=a$
and all colours are used since $t=\nu_{3s}+2=4e(B)+2$.

It remains to verify that $A$ has the required absorbing property. Take $Z'\subseteq Z\sm \{z_1,z_2\}$ of size $\beta n$ and let $Z'_B\subseteq Z_B$ be the corresponding vertex subset of $B$. By Lemma~\ref{lemma:rmbg}, $B$ has a perfect matching $M=\{x_1w_1,\ldots,x_{3s}w_{3s}\}$ covering $X_B\cup Y_B\cup Z_B'$. For each $i$, let $w'_i\in L_i\subset Y\cup Z'$ be the vertex corresponding to $w_i\in Y_B\cup Z'_B$. Then since $F^i$ is an $L_i$-absorbing gadget, there is a rainbow path $P_i\subseteq \mbf G^i$ between $b_-^i$ and $b_{+}^i$ covering $(V(F^i)\sm L_i)\cup \{w'_i\}$ and using colours $\mu_i,\mu_i+1,\ldots,\nu_i$ in order. Taking all these paths together with the already fixed cherries, we obtain a $z_1,z_2$-path whose internal vertices cover $A\cup Z'$ exactly.  \end{proof}

\section{Almost spanning path cover} \label{sec:path cover}
In this section, we prove Lemma \ref{lemma:pathcoll} giving a collection of paths covering almost all of the vertex set of our graph. 

\begin{proof}[Proof of Lemma \ref{lemma:pathcoll}]
We will use the following random greedy algorithm to obtain the required paths.
First, remove at most one vertex from each $V_j$ so that they all have size $n_1 := \lfloor n/K\rfloor$.
Moreover,   let  $s \leq (1-\eps)n/K$ as given in the statement of the lemma.

\vspace{5pt}
  \begin{algorithm}[H]
  %  \SetAlgoNoEnd
    \caption{Path builder}\label{alg:path}
  Initialise by fixing $V^1_j:=V_j$  for $j\in [K]$ and $n_1$ as above
  \\
    \For{$i = 1, 2, \dotsc, s$}{    
      \eIf{$\delta(\mathbf{G}[V^i_j,V^i_{j+1}])< n_i/2$ for some $j\in[K-1]$}{
        Abort and return an error.
      }{
      \For{$j=1,2,\dotsc,K-1$} {
      Let $M_j^i$ be an arbitrary perfect matching in $G_{\chi_i(j)}[V_j^i,V_{j+1}^i]$ (which exists by Lemma~\ref{lem:hall} as we did not abort)}
      Consider the union $\bigcup_j M_j^i$ of these perfect matchings, which consists of $n_i$ vertex-disjoint exactly $\chi_i$-coloured paths, each of length $K-1$.\\
Choose one path uniformly at random and label it $P_i$. \\
Let $V_j^{i+1} := V_j^i \sm V(P_i)$ for all $j \in [K]$ and let $n_{i+1} := n_i-1$. 
         }
    }
  \end{algorithm}
  \vspace{5pt}
Note that this simple algorithm succeeds in producing the required family of paths $P_1,\ldots, P_s$ as long as the algorithm does not abort at any step $i$. We will now prove that with high probability this is indeed the case. 

In order to analyse this algorithm, we use a nibble approach and group together the steps of the algorithm into small collections, each of which we will show behaves well. To this end, we define $\gamma := \eps\alpha/(3K)$ and $t:=\lceil{s/(\gamma n)}\rceil$. Moreover, for $\ell=1,\ldots, t$, let $i_\ell:=(\ell-1)\gamma n+1$. 
%It suffices to show that the algorithm does not abort, as then the required collection of paths is $\mathcal{P}_1 \cup \ldots \cup \mathcal{P}_{s}$.
Finally, let  $\delta := 1-((1+\alpha)/(1+2\alpha))^{1/t}$, 
and note that
\begin{equation} \label{eq:deltabd}
(1-\delta)^\ell(1/2+\alpha) \geq 1/2+\alpha/2
\quad\text{for all }\ell \in [t]. \end{equation}

Now for $\ell\in [t]$,
 let 
$\mathcal{B}_\ell$ be the bad event that $\delta(\mathbf{G}[V_j^{i_\ell},V_{j+1}^{i_\ell}])<(1-\delta)^\ell(1/2+\alpha)n_{i_\ell}$ for some $j \in [K-1]$.
Define a stopping time $T$ to be the minimum $\ell$ such that $\mathcal{B}_\ell$ holds. We will show that, with high probability, $T>t$.

First we claim that, for $\ell\in [t]$, if $T>\ell$, then the algorithm does not abort in steps $i=i_\ell,\ldots,\min\{i_{\ell+1}-1,s\}$.
Indeed, for any $i_\ell \leq i < i_{\ell+1}$, we have
%\begin{linenomath}
    \begin{equation*}
\delta(\mathbf{G}[V_j^i,V_{j+1}^i]) \geq  \delta(\mathbf{G}[V_j^{i_\ell},V_{j+1}^{i_\ell}])-(i-i_\ell) \geq (1/2+\alpha/2)n_{i_\ell}-\gamma n>
n_{i_\ell}/2 \geq n_i/2,
\end{equation*} 
%\end{linenomath}
since 
\begin{equation} \label{eq:ni}
n_{i_\ell} = \lfloor n/K\rfloor - (i_\ell-1) \geq n/K - (1-\eps)n/K = \eps n/K > 2\gamma n/\alpha.    \end{equation}
%Thus Dirac's theorem implies that we can complete Substep $i$ and find the required perfect matching $M_j^i$.
This shows in particular that if $T>t$, then the algorithm succeeds at every step and never aborts.

    We now show that with high probability, $T>t$. By our minimum degree assumption, note that $T>1$. For further steps, we upper bound $\mathbb{P}(T=\ell+1 \ | \ T>\ell)$ for $\ell=1,\ldots,t-1$.  Fixing one such $\ell$, 
let $Y^\ell_j := (V(P_{i_\ell}) \cup \ldots \cup V(P_{i_{\ell+1}-1})) \cap V_j$ which, as a function of the random choices made in steps~$i_\ell,\ldots, i_{\ell+1}-1$,
is a random variable. Moreover, under the condition that $T>\ell$, as shown above, the algorithm will not abort in these steps and so this random variable is well-defined. 
Finally, note that under the condition that $T>\ell$ we have that for all $j \in [K]$,  $Y^\ell_j$ is a uniform random set chosen among all subsets of $V_j^{i_\ell}$ of size $\gamma n$.
(Note, though, that $Y^\ell_j$ and $Y^\ell_{j'}$ are dependent for different $j,j'$.)

Therefore, for each $c \in [n]$, $j \in [K]$, $x \in V_j^{i_\ell}$ and $g \in \{+1,-1\}$, the random variable $d^g_c(x) := |N_{G_c}(x,V^{i_\ell}_{j+g}) \sm Y^\ell_{j+g}|$
has hypergeometric distribution (we do not define this for $j+g=0,K+1$).
We have
    %\begin{linenomath} 
    \begin{equation*}
\mathbb{E}(d^g_c(x) | \neg \mc{B}_\ell) = \mathbb{E}\left(\left(\frac{n_{i_\ell}-\gamma n}{n_{i_\ell}}\right)d_{G_c}(x,V^{i_\ell}_{j+g})~ \middle\vert ~\neg \mc{B}_\ell \right)
\geq (1-\delta)^\ell(1/2+\alpha)n_{i_{\ell+1}}.
         \end{equation*}
%\end{linenomath}
Thus Lemma~\ref{lem:chernoff} (Chernoff bounds) implies that
 %\begin{linenomath}
 \begin{align*}
\mathbb{P}\left(d^g_c(x) < (1-\delta)^{\ell+1}(1/2+\alpha)n_{i_{\ell+1}} ~\middle\vert~ \neg \mc{B}_\ell\right) 
&\leq 2\exp\left(-\delta^2(1-\delta)^\ell(1/2+\alpha)n_{i_{\ell+1}}/3\right) \\&\leq 2\exp(-\delta^2\gamma n/(3\alpha)) \leq \exp(-\sqrt{n}),
\end{align*} 
using \eqref{eq:deltabd} and \eqref{eq:ni} in the penultimate inequality here.
%\end{linenomath}
Thus, taking a union bound over at most $n\cdot n \cdot 2$  events, corresponding to a choice of $x,c$ and $g$, we get that  $\mathbb{P}(T=\ell+1 \ | \ T>\ell) \leq \exp(-n^{1/3})$.

Finally, we have that 
\[\mathbb{P}(T\leq t)=\sum_{\ell=0}^{t-1}\mathbb{P}(T=\ell+1 \ | \ T>\ell)\leq t\exp(-n^{1/3})=o(1),\]
using our estimate above and the fact that $\mathbb{P}(T=1)=0$.  This completes the proof, showing with high probability the algorithm succeeds at every step and produces the required collection of paths $P_1,\ldots,P_s$. 
\end{proof}

\section{Concluding remarks} \label{sec:conclude}

Given an extremal threshold for a certain subgraph with $m$ edges, one can ask whether the same (asymptotic) condition applied to a family of $m$ graphs $\mbf G$ guarantees the existence of a transversal copy of that subgraph. This question was posed by Joos and Kim \cite{jooskim} and formalised by Gupta, Hamann, M\"uyesser, Parczyk and Sgueglia \cite{ghmps} through the notion of \emph{colour-blindness}. Following our work here, one can ask whether this condition in fact guarantees universality for transversal copies of the subgraph in question.

Formalising slightly, one can define the \emph{transversal threshold} for a given $m$-edge subgraph $J$ to be the optimal extremal condition (for example a minimum degree or density condition) which, when applied to each graph of an $m$-graph collection $\mbf G$, guarantees a transversal copy of $J$, and the \emph{pattern threshold} to be the optimal extremal condition guaranteeing \emph{all possible} rainbow colour patterns of $J$. Clearly the pattern threshold is always at least as large as the transversal threshold, which is in turn at least as large as the classical extremal threshold (by considering many copies of the same host graph as in the introduction). In this terminology, Joos and Kim \cite{jooskim} showed that for Hamiltonicity, the extremal and transversal minimum degree thresholds coincide exactly and our result here shows that the pattern threshold is also asymptotically the same. On the other hand, the construction given in Section \ref{sec:counterex} shows that the pattern threshold does not coincide exactly with the other two thresholds. 
It would be interesting to establish the location of more transversal and pattern thresholds for different subgraphs and whether they (exactly or asymptotically) coincide, or if there is separation.

As discussed in the introduction, previous work shows that for many spanning structures in graphs and hypergraphs, the extremal and transversal thresholds asymptotically coincide. Montgomery, M\"uyesser and Pehova \cite{mmp} give examples where the pattern and transversal thresholds for graph tilings are asymptotically the same, as well as when they are separated. They also give examples of tilings for which there is a separation between the extremal and transversal minimum degree thresholds, as is the case with the edge density threshold for a single triangle. Indeed, whilst Mantel's theorem gives that the extremal edge density for containing a triangle is $1/2$, Aharoni, DeVos, de la Maza, Montejano and \v{S}\'amal~\cite{Aharoni} showed that the transversal threshold is much higher. 
As all rainbow colourings of a triangle are isomorphic, the transversal and pattern thresholds trivially coincide. 

Before discussing further research directions, we also mention that these sorts of questions also make sense in non-extremal settings. Indeed, for any condition on a host graph which guarantees the existence of a certain subgraph, one can ask whether applying such a condition to each graph of a graph collection results in (all possible) transversal copies of the desired subgraph, or not. 

\subsection{Improving linear error}
In this paper, we showed that when a collection of graphs is asymptotically Dirac, one can find a Hamilton cycle whose edges come from the collection of graphs in any fixed order. Our proof requires the additional $o(n)$ term in the minimum degree of each graph, though we suspect that one could lower this error to a logarithmic or even large constant term, using some iterative ideas and more careful analysis. However, we believe that the construction in Section \ref{sec:counterex} yields the exact threshold, and in light of this it seems more pertinent to either obtain results that improve this $o(n)$ term to an explicit and relatively small constant, or find counterexamples to this being possible.

\subsection{Extensions}
There are many natural questions that relate to our work here, and for which we believe that the proof methods would be useful. We note several such extensions.

Considering Hamilton cycles in hypergraphs, Cheng, Han, Wang, Wang and Yang~\cite{chwwy} showed that $(k-1)$-degree $(1/2+o(1))n$ is sufficient to guarantee the existence of a rainbow tight Hamilton cycle in a $k$-uniform $n$-vertex hypergraph collection. We believe that the ideas behind our absorbing gadgets extend directly to this setting, but that one needs a slightly more involved strategy to obtain something equivalent to our cherry connections between pairs of fixed $(k-1)$-tuples, and to handle the `almost cover'. One could also consider Hamilton $\ell$-cycles and $d$-degree conditions for other values of $\ell$ and $d$ between $1$ and $k-1$, where our methods look to be useful, but additional ideas are needed to obtain tight results. 

Returning to the setting of graphs, beyond Hamilton cycles, one can ask whether our strategy extends to find bounded degree trees in any colour pattern. If the tree contains a long bare path, our absorption goes through immediately, and Lemma \ref{lemma:degenk} can in fact be used to embed absorbing gadgets for some substructures other than just paths. It would also be natural to consider powers of Hamilton cycles.

\section{Acknowledgements}

The work leading to these results was carried out at a workshop funded by the ERC Starting Grant 947978 of Richard Montgomery and hosted at the University of Warwick. We thank them for their hospitality and for providing a stimulating and enjoyable research environment, and the ERC for their support of this workshop.

We thank Alp M\"uyesser for suggesting this interesting problem at the \emph{UCL Workshop in Extremal and Probabilistic Combinatorics}; Mat\'{i}as Pavez-Sign\'{e} for helpful discussions;
and the referees for their careful reading and useful comments.

\bibliographystyle{abbrv}
\bibliography{transversal}

\end{document}

%% file: absorbing_gadget.tex
\newrgbcolor{wqwqwq}{0.3764705882352941 0.3764705882352941 0.3764705882352941}
\newrgbcolor{ffxfqq}{1. 0.4980392156862745 0.}
\newrgbcolor{xfqqff}{0.4980392156862745 0. 1.}
\newrgbcolor{ffdxqq}{1. 0.8431372549019608 0.}
\definecolor{bfffqq}{rgb}{0.7490196078431373 1. 0.}
\newrgbcolor{ffzzzz}{1. 0.6 0.6}
\newrgbcolor{zzccff}{0.6 0.8 1.}
\newrgbcolor{qqwwzz}{0. 0.4 0.6}
\newrgbcolor{zzwwqq}{0.6 0.4 0.}
\newrgbcolor{ududff}{0.30196078431372547 0.30196078431372547 1.}
\psset{xunit=0.8cm,yunit=0.8cm,algebraic=true,dimen=middle,dotstyle=o,dotsize=5pt 0,linewidth=2.pt,arrowsize=3pt 2,arrowinset=0.25}
\begin{pspicture*}(-3.4451474065047036,-6.770079881618745)(17.413297544696626,3.0444734432460216)
\psline[linewidth=2.pt,linecolor=red](0.,0.)(1.,-2.)
\psline[linewidth=2.pt,linecolor=blue](1.,-2.)(2.,0.)
\psline[linewidth=2.pt,linecolor=ffxfqq](2.,0.)(3.,0.)
\psline[linewidth=2.pt,linecolor=xfqqff](3.,0.)(4.,0.)
\psline[linewidth=2.pt,linecolor=ffdxqq](4.,0.)(5.,-2.)
\psline[linewidth=2.pt,linecolor=magenta](5.,-2.)(6.,0.)
\psline[linewidth=2.pt,linecolor=bfffqq](6.,0.)(7.,0.)
\psline[linewidth=2.pt,linecolor=cyan](7.,0.)(8.,0.)
\psline[linewidth=2.pt,linecolor=green](8.,0.)(9.,-2.)
\psline[linewidth=2.pt,linecolor=ffzzzz](9.,-2.)(10.,0.)
\psline[linewidth=2.pt,linecolor=zzccff](10.,0.)(11.,0.)
\psline[linewidth=2.pt,linecolor=RawSienna](11.,0.)(12.,0.)
\psline[linewidth=2.pt,linecolor=qqwwzz](12.,0.)(13.,-2.)
\psline[linewidth=2.pt,linecolor=zzwwqq](13.,-2.)(14.,0.)
\psline[linewidth=2.pt,linecolor=qqwwzz](11.,2.)(12.,0.)
\psline[linewidth=2.pt,linecolor=zzwwqq](11.,2.)(14.,0.)
\psline[linewidth=2.pt,linecolor=ffzzzz](10.,0.)(11.,2.)
\psline[linewidth=2.pt,linecolor=green](8.,0.)(11.,2.)
\psline[linewidth=2.pt,linecolor=ffzzzz](7.,2.)(10.,0.)
\psline[linewidth=2.pt,linecolor=green](7.,2.)(8.,0.)
\psline[linewidth=2.pt,linecolor=magenta](6.,0.)(7.,2.)
\psline[linewidth=2.pt,linecolor=ffdxqq](4.,0.)(7.,2.)
\psline[linewidth=2.pt,linecolor=magenta](3.,2.)(6.,0.)
\psline[linewidth=2.pt,linecolor=ffdxqq](3.,2.)(4.,0.)
\psline[linewidth=2.pt,linecolor=blue](3.,2.)(2.,0.)
\psline[linewidth=2.pt,linecolor=red](3.,2.)(0.,0.)
\psline[linewidth=2.pt](-5.,-10.)(66.8041152542315,-9.927882152908712)
\psline[linewidth=2.pt,linecolor=red](-2.5,-5.)(-2.,-6.)
\psline[linewidth=2.pt,linecolor=blue](-2.,-6.)(-1.5,-5.)
\psline[linewidth=2.pt,linecolor=ffxfqq](-1.5,-5.)(-1.,-5.)
\psline[linewidth=2.pt,linecolor=xfqqff](-1.,-5.)(-0.5,-5.)
\psline[linewidth=1.2pt,linestyle=dashed,dash=3pt 3pt,linecolor=lightgray](-0.5,-5.)(0.,-6.)
\psline[linewidth=1.2pt,linestyle=dashed,dash=3pt 3pt,linecolor=lightgray](0.,-6.)(0.5,-5.)
\psline[linewidth=2.pt,linecolor=bfffqq](0.5,-5.)(1.,-5.)
\psline[linewidth=2.pt,linecolor=cyan](1.,-5.)(1.5,-5.)
\psline[linewidth=1.2pt,linestyle=dashed,dash=3pt 3pt,linecolor=lightgray](1.5,-5.)(2.,-6.)
\psline[linewidth=1.2pt,linestyle=dashed,dash=3pt 3pt,linecolor=lightgray](2.,-6.)(2.5,-5.)
\psline[linewidth=2.pt,linecolor=zzccff](2.5,-5.)(3.,-5.)
\psline[linewidth=2.pt,linecolor=RawSienna](3.,-5.)(3.5,-5.)
\psline[linewidth=1.2pt,linestyle=dashed,dash=3pt 3pt,linecolor=lightgray](3.5,-5.)(4.,-6.)
\psline[linewidth=1.2pt,linestyle=dashed,dash=3pt 3pt,linecolor=lightgray](4.,-6.)(4.5,-5.)
\psline[linewidth=2.pt,linecolor=qqwwzz](3.,-4.)(3.5,-5.)
\psline[linewidth=2.pt,linecolor=zzwwqq](3.,-4.)(4.5,-5.)
\psline[linewidth=1.2pt,linestyle=dashed,dash=3pt 3pt,linecolor=lightgray](2.5,-5.)(3.,-4.)
\psline[linewidth=1.2pt,linestyle=dashed,dash=3pt 3pt,linecolor=lightgray](1.5,-5.)(3.,-4.)
\psline[linewidth=2.pt,linecolor=ffzzzz](1.,-4.)(2.5,-5.)
\psline[linewidth=2.pt,linecolor=green](1.,-4.)(1.5,-5.)
\psline[linewidth=1.2pt,linestyle=dashed,dash=3pt 3pt,linecolor=lightgray](0.5,-5.)(1.,-4.)
\psline[linewidth=1.2pt,linestyle=dashed,dash=3pt 3pt,linecolor=lightgray](-0.5,-5.)(1.,-4.)
\psline[linewidth=2.pt,linecolor=magenta](-1.,-4.)(0.5,-5.)
\psline[linewidth=2.pt,linecolor=ffdxqq](-1.,-4.)(-0.5,-5.)
\psline[linewidth=1.2pt,linestyle=dashed,dash=3pt 3pt,linecolor=lightgray](-1.,-4.)(-1.5,-5.)
\psline[linewidth=1.2pt,linestyle=dashed,dash=3pt 3pt,linecolor=lightgray](-1.,-4.)(-2.5,-5.)
\psline[linewidth=1.2pt,linestyle=dashed,dash=3pt 3pt,linecolor=lightgray](9.46615580711768,-4.987981560782237)(9.96615580711768,-5.987981560782237)
\psline[linewidth=1.2pt,linestyle=dashed,dash=3pt 3pt,linecolor=lightgray](9.96615580711768,-5.987981560782237)(10.46615580711768,-4.987981560782237)
\psline[linewidth=2.pt,linecolor=ffxfqq](10.46615580711768,-4.987981560782237)(10.96615580711768,-4.987981560782237)
\psline[linewidth=2.pt,linecolor=xfqqff](10.96615580711768,-4.987981560782237)(11.46615580711768,-4.987981560782237)
\psline[linewidth=2.pt,linecolor=ffdxqq](11.46615580711768,-4.987981560782237)(11.96615580711768,-5.987981560782237)
\psline[linewidth=2.pt,linecolor=magenta](11.96615580711768,-5.987981560782237)(12.46615580711768,-4.987981560782237)
\psline[linewidth=2.pt,linecolor=bfffqq](12.46615580711768,-4.987981560782237)(12.96615580711768,-4.987981560782237)
\psline[linewidth=2.pt,linecolor=cyan](12.96615580711768,-4.987981560782237)(13.46615580711768,-4.987981560782237)
\psline[linewidth=1.2pt,linestyle=dashed,dash=3pt 3pt,linecolor=lightgray](13.46615580711768,-4.987981560782237)(13.96615580711768,-5.987981560782237)
\psline[linewidth=1.2pt,linestyle=dashed,dash=3pt 3pt,linecolor=lightgray](13.96615580711768,-5.987981560782237)(14.46615580711768,-4.987981560782237)
\psline[linewidth=2.pt,linecolor=zzccff](14.46615580711768,-4.987981560782237)(14.96615580711768,-4.987981560782237)
\psline[linewidth=2.pt,linecolor=RawSienna](14.96615580711768,-4.987981560782237)(15.46615580711768,-4.987981560782237)
\psline[linewidth=1.2pt,linestyle=dashed,dash=3pt 3pt,linecolor=lightgray](15.46615580711768,-4.987981560782237)(15.96615580711768,-5.987981560782237)
\psline[linewidth=1.2pt,linestyle=dashed,dash=3pt 3pt,linecolor=lightgray](15.96615580711768,-5.987981560782237)(16.46615580711768,-4.987981560782237)
\psline[linewidth=2.pt,linecolor=qqwwzz](14.96615580711768,-3.987981560782237)(15.46615580711768,-4.987981560782237)
\psline[linewidth=2.pt,linecolor=zzwwqq](14.96615580711768,-3.987981560782237)(16.46615580711768,-4.987981560782237)
\psline[linewidth=1.2pt,linestyle=dashed,dash=3pt 3pt,linecolor=lightgray](14.46615580711768,-4.987981560782237)(14.96615580711768,-3.987981560782237)
\psline[linewidth=1.2pt,linestyle=dashed,dash=3pt 3pt,linecolor=lightgray](13.46615580711768,-4.987981560782237)(14.96615580711768,-3.987981560782237)
\psline[linewidth=2.pt,linecolor=ffzzzz](12.96615580711768,-3.987981560782237)(14.46615580711768,-4.987981560782237)
\psline[linewidth=2.pt,linecolor=green](12.96615580711768,-3.987981560782237)(13.46615580711768,-4.987981560782237)
\psline[linewidth=1.2pt,linestyle=dashed,dash=3pt 3pt,linecolor=lightgray](12.46615580711768,-4.987981560782237)(12.96615580711768,-3.987981560782237)
\psline[linewidth=1.2pt,linestyle=dashed,dash=3pt 3pt,linecolor=lightgray](11.46615580711768,-4.987981560782237)(12.96615580711768,-3.987981560782237)
\psline[linewidth=1.2pt,linestyle=dashed,dash=3pt 3pt,linecolor=lightgray](10.96615580711768,-3.987981560782237)(12.46615580711768,-4.987981560782237)
\psline[linewidth=1.2pt,linestyle=dashed,dash=3pt 3pt,linecolor=lightgray](10.96615580711768,-3.987981560782237)(11.46615580711768,-4.987981560782237)
\psline[linewidth=2.pt,linecolor=blue](10.96615580711768,-3.987981560782237)(10.46615580711768,-4.987981560782237)
\psline[linewidth=2.pt,linecolor=red](10.96615580711768,-3.987981560782237)(9.46615580711768,-4.987981560782237)
\begin{scriptsize}
\psdots[dotstyle=*,linecolor=wqwqwq](0.,0.)
\rput[bl](-0.22164280263561784,0.2028881890243951){\wqwqwq{$b_0$}}
\psdots[dotstyle=*,linecolor=wqwqwq](2.,0.)
\rput[bl](1.676798406663571,0.2028881890243951){\wqwqwq{$b_1$}}
\psdots[dotstyle=*,linecolor=wqwqwq](3.,0.)
\rput[bl](3.0844527521322345,0.2028881890243951){\wqwqwq{$b_2$}}
\psdots[dotstyle=*,linecolor=wqwqwq](4.,0.)
\rput[bl](3.871954010691524,0.2028881890243951){\wqwqwq{$b_3$}}
\psdots[dotstyle=*,linecolor=wqwqwq](6.,0.)
\rput[bl](5.687262701628851,0.2028881890243951){\wqwqwq{$b_4$}}
\psdots[dotstyle=*,linecolor=wqwqwq](7.,0.)
\rput[bl](7.0747639601881405,0.2028881890243951){\wqwqwq{$b_5$}}
\psdots[dotstyle=*,linecolor=wqwqwq](8.,0.)
\rput[bl](7.882418305656804,0.2028881890243951){\wqwqwq{$b_6$}}
\psdots[dotstyle=*,linecolor=wqwqwq](10.,0.)
\rput[bl](9.677573909684757,0.2028881890243951){\wqwqwq{$b_7$}}
\psdots[dotstyle=*,linecolor=wqwqwq](11.,0.)
\rput[bl](11.085228255153421,0.2028881890243951){\wqwqwq{$b_8$}}
\psdots[dotstyle=*,linecolor=wqwqwq](12.,0.)
\rput[bl](12.07272951371271,0.2028881890243951){\wqwqwq{$b_9$}}
\psdots[dotstyle=*,linecolor=wqwqwq](14.,0.)
\rput[bl](14.088038204650037,0.2028881890243951){\wqwqwq{$b_{10}$}}
\psdots[dotstyle=*,linecolor=wqwqwq](3.,2.)
\rput[bl](3.0844527521322345,2.1980437930523458){\wqwqwq{$c_1$}}
\psdots[dotstyle=*,linecolor=wqwqwq](7.,2.)
\rput[bl](7.0747639601881405,2.1980437930523458){\wqwqwq{$c_2$}}
\psdots[dotstyle=*,linecolor=wqwqwq](11.,2.)
\rput[bl](11.085228255153421,2.1980437930523458){\wqwqwq{$c_3$}}
\psdots[dotstyle=*,linecolor=wqwqwq](1.,-2.)
\rput[bl](1.2102156695605206,-2.2356353270097666){\wqwqwq{$a_1$}}
\psdots[dotstyle=*,linecolor=wqwqwq](5.,-2.)
\rput[bl](5.2005268776164275,-2.275941500828513){\wqwqwq{$a_2$}}
\psdots[dotstyle=*,linecolor=wqwqwq](9.,-2.)
\rput[bl](9.190838085672334,-2.2356353270097666){\wqwqwq{$a_3$}}
\psdots[dotstyle=*,linecolor=wqwqwq](13.,-2.)
\rput[bl](13.221455467546987,-2.275941500828513){\wqwqwq{$a_4$}}
\psdots[dotstyle=*,linecolor=ududff](-5.,-10.)
\psdots[dotstyle=*,linecolor=ududff](66.8041152542315,-9.927882152908712)
\psdots[dotstyle=*,linecolor=ududff](18.93231161423536,-9.975963121564474)
\psdots[dotstyle=*,linecolor=ududff](42.86462322847072,-9.951926243128948)
\psdots[dotsize=3pt 0,dotstyle=*,linecolor=wqwqwq](-2.,-6.)
\psdots[dotsize=3pt 0,dotstyle=*,linecolor=wqwqwq](-2.,-6.)
\psdots[dotsize=3pt 0,dotstyle=*,linecolor=wqwqwq](-1.,-5.)
\psdots[dotsize=3pt 0,dotstyle=*,linecolor=wqwqwq](0.,-6.)
\psdots[dotsize=3pt 0,dotstyle=*,linecolor=wqwqwq](0.,-6.)
\psdots[dotsize=3pt 0,dotstyle=*,linecolor=wqwqwq](1.,-5.)
\psdots[dotsize=3pt 0,dotstyle=*,linecolor=wqwqwq](2.,-6.)
\psdots[dotsize=3pt 0,dotstyle=*,linecolor=wqwqwq](2.,-6.)
\psdots[dotsize=3pt 0,dotstyle=*,linecolor=wqwqwq](3.,-5.)
\psdots[dotsize=3pt 0,dotstyle=*,linecolor=wqwqwq](4.,-6.)
\psdots[dotsize=3pt 0,dotstyle=*,linecolor=wqwqwq](4.,-6.)
\psdots[dotsize=3pt 0,dotstyle=*,linecolor=wqwqwq](3.,-4.)
\psdots[dotsize=3pt 0,dotstyle=*,linecolor=wqwqwq](3.5,-5.)
\psdots[dotsize=2pt 0,dotstyle=*,linecolor=wqwqwq](3.,-4.)
\psdots[dotsize=3pt 0,dotstyle=*,linecolor=wqwqwq](4.5,-5.)
\psdots[dotsize=2pt 0,dotstyle=*,linecolor=wqwqwq](3.,-4.)
\psdots[dotsize=2pt 0,dotstyle=*,linecolor=wqwqwq](3.,-4.)
\psdots[dotsize=3pt 0,dotstyle=*,linecolor=wqwqwq](1.,-4.)
\psdots[dotsize=3pt 0,dotstyle=*,linecolor=wqwqwq](2.5,-5.)
\psdots[dotsize=2pt 0,dotstyle=*,linecolor=wqwqwq](1.,-4.)
\psdots[dotsize=3pt 0,dotstyle=*,linecolor=wqwqwq](1.5,-5.)
\psdots[dotsize=2pt 0,dotstyle=*,linecolor=wqwqwq](1.,-4.)
\psdots[dotsize=2pt 0,dotstyle=*,linecolor=wqwqwq](1.,-4.)
\psdots[dotsize=3pt 0,dotstyle=*,linecolor=wqwqwq](-1.,-4.)
\psdots[dotsize=3pt 0,dotstyle=*,linecolor=wqwqwq](0.5,-5.)
\psdots[dotsize=2pt 0,dotstyle=*,linecolor=wqwqwq](-1.,-4.)
\psdots[dotsize=3pt 0,dotstyle=*,linecolor=wqwqwq](-0.5,-5.)
\psdots[dotsize=2pt 0,dotstyle=*,linecolor=wqwqwq](-1.,-4.)
\psdots[dotsize=3pt 0,dotstyle=*,linecolor=wqwqwq](-1.5,-5.)
\psdots[dotsize=2pt 0,dotstyle=*,linecolor=wqwqwq](-1.,-4.)
\psdots[dotsize=3pt 0,dotstyle=*,linecolor=wqwqwq](-2.5,-5.)
\psdots[dotsize=3pt 0,dotstyle=*,linecolor=wqwqwq](9.96615580711768,-5.987981560782237)
\psdots[dotsize=3pt 0,dotstyle=*,linecolor=wqwqwq](9.96615580711768,-5.987981560782237)
\psdots[dotsize=3pt 0,dotstyle=*,linecolor=wqwqwq](10.96615580711768,-4.987981560782237)
\psdots[dotsize=3pt 0,dotstyle=*,linecolor=wqwqwq](11.96615580711768,-5.987981560782237)
\psdots[dotsize=3pt 0,dotstyle=*,linecolor=wqwqwq](11.96615580711768,-5.987981560782237)
\psdots[dotsize=3pt 0,dotstyle=*,linecolor=wqwqwq](12.96615580711768,-4.987981560782237)
\psdots[dotsize=3pt 0,dotstyle=*,linecolor=wqwqwq](13.96615580711768,-5.987981560782237)
\psdots[dotsize=3pt 0,dotstyle=*,linecolor=wqwqwq](13.96615580711768,-5.987981560782237)
\psdots[dotsize=3pt 0,dotstyle=*,linecolor=wqwqwq](14.96615580711768,-4.987981560782237)
\psdots[dotsize=3pt 0,dotstyle=*,linecolor=wqwqwq](15.96615580711768,-5.987981560782237)
\psdots[dotsize=3pt 0,dotstyle=*,linecolor=wqwqwq](15.96615580711768,-5.987981560782237)
\psdots[dotsize=3pt 0,dotstyle=*,linecolor=wqwqwq](14.96615580711768,-3.987981560782237)
\psdots[dotsize=3pt 0,dotstyle=*,linecolor=wqwqwq](15.46615580711768,-4.987981560782237)
\psdots[dotsize=3pt 0,dotstyle=*,linecolor=wqwqwq](14.96615580711768,-3.987981560782237)
\psdots[dotsize=3pt 0,dotstyle=*,linecolor=wqwqwq](16.46615580711768,-4.987981560782237)
\psdots[dotsize=3pt 0,dotstyle=*,linecolor=wqwqwq](14.96615580711768,-3.987981560782237)
\psdots[dotsize=3pt 0,dotstyle=*,linecolor=wqwqwq](13.46615580711768,-4.987981560782237)
\psdots[dotsize=3pt 0,dotstyle=*,linecolor=wqwqwq](14.96615580711768,-3.987981560782237)
\psdots[dotsize=3pt 0,dotstyle=*,linecolor=wqwqwq](12.96615580711768,-3.987981560782237)
\psdots[dotsize=4pt 0,dotstyle=*,linecolor=wqwqwq](14.46615580711768,-4.987981560782237)
\psdots[dotsize=3pt 0,dotstyle=*,linecolor=wqwqwq](12.96615580711768,-3.987981560782237)
\psdots[dotsize=3pt 0,dotstyle=*,linecolor=wqwqwq](12.96615580711768,-3.987981560782237)
\psdots[dotsize=3pt 0,dotstyle=*,linecolor=wqwqwq](12.96615580711768,-3.987981560782237)
\psdots[dotsize=3pt 0,dotstyle=*,linecolor=wqwqwq](10.96615580711768,-3.987981560782237)
\psdots[dotsize=3pt 0,dotstyle=*,linecolor=wqwqwq](12.46615580711768,-4.987981560782237)
\psdots[dotsize=3pt 0,dotstyle=*,linecolor=wqwqwq](10.96615580711768,-3.987981560782237)
\psdots[dotsize=3pt 0,dotstyle=*,linecolor=wqwqwq](11.46615580711768,-4.987981560782237)
\psdots[dotsize=3pt 0,dotstyle=*,linecolor=wqwqwq](10.96615580711768,-3.987981560782237)
\psdots[dotsize=3pt 0,dotstyle=*,linecolor=wqwqwq](10.46615580711768,-4.987981560782237)
\psdots[dotsize=3pt 0,dotstyle=*,linecolor=wqwqwq](10.96615580711768,-3.987981560782237)
\psdots[dotsize=3pt 0,dotstyle=*,linecolor=wqwqwq](9.46615580711768,-4.987981560782237)
\end{scriptsize}
\end{pspicture*}

%% file: absorption.tex
\newrgbcolor{ffttww}{1. 0.2 0.4}
\newrgbcolor{uququq}{0.25 0.25 0.25}
\newrgbcolor{mygreen}{0.4 1 0}
\newrgbcolor{myteal}{0.2 1. 0.8}
\psset{xunit=0.5cm,yunit=0.5cm,algebraic=true,dimen=middle,dotstyle=o,dotsize=5pt 0,linewidth=2.pt,arrowsize=3pt 2,arrowinset=0.25}
\begin{pspicture*}(-14.523025407626294,-2.890510721352683)(14.405075454007827,6.416554654638744)
\pspolygon[linewidth=0.8pt,linestyle=dashed,dash=3pt 3pt,linecolor=uququq](-13.6,3.2)(-6.7,3.2)(-6.7,4.9)(-13.6,4.8)
\pspolygon[linewidth=0.8pt,linestyle=dashed,dash=3pt 3pt,linecolor=uququq](-12.8,-0.6)(-8.1,-0.6)(-8.1,0.7)(-12.8,0.7)

\psset{plotpoints=300,linewidth=1.6pt}
\psparametricplotHSB[algebraic,HueBegin=0,HueEnd=0.5]{-4}{4}{ t | 2*sin(5*t)/(5*(t^2)^(1/4))}
\psparametricplotHSB[algebraic,HueBegin=0.5,HueEnd=1]{5}{13}{ t | 2*sin(5*(t-9))/(5*((t-9)^2)^(1/4))}

\rput{0.}(0.9991926854432662,4.){\psellipse[linewidth=0.8pt,linecolor=uququq](0,0)(3.901113819873431,0.9823453475335086)}
\rput{0.}(8.60803389916491,4.){\psellipse[linewidth=0.8pt,linecolor=uququq](0,0)(5.5205331437195095,1.184477842068631)}
\rput[tl](-2.657877735309759,6.117204598685803){$N_B(x_i)$}
\rput[tl](10.438788535928694,6.089341013188477){$N_B(x_{i+1})$}
\rput[tl](-1.4716910314709785,-1.665896856090653){$N_B(x_i)$-absorber}
\rput[tl](7.316124502807845,-1.6386832146403858){$N_B(x_{i+1})$-absorber}
\newrgbcolor{myblue}{0.2 0.2 1.}
\psline[linecolor=myblue, linewidth=1.6pt](4.,4.)(7.743362938564151,0.)
\newrgbcolor{myviolet}{0.4980392156862745 0. 1.}
\psline[linecolor=myviolet,linewidth=1.6pt](4.,4.)(7.909204383634306,0.31055636004788233)
\newrgbcolor{mygreen}{0.7490196078431373 1. 0.}
\psline[linecolor=mygreen,linewidth=1.6pt](-2.,4.)(-1.0370687869746776,0.3874887971621966)
\newrgbcolor{myyellow}{1 1 0}
\psline[linecolor=myyellow,linewidth=1.6pt](-2.,4.)(-1.3189702925479404,-0.054681504656606)
\psline[linewidth=1.6pt,linecolor=uququq](-12.637948160495183,4.)(-11.645987966519298,0.)
\psline[linewidth=1.6pt,linestyle=dashed,dash=3pt 3pt,linecolor=uququq](-11.645987966519298,0.)(-11.650431064754812,4.)
\psline[linewidth=1.6pt,linestyle=dashed,dash=3pt 3pt,linecolor=uququq](-11.645987966519298,0.)(-10.662913969014442,4.)
\psline[linewidth=1.6pt,linecolor=uququq](-10.662913969014442,4.)(-9.192114958332118,0.)
\psline[linewidth=1.6pt,linestyle=dashed,dash=3pt 3pt,linecolor=uququq](-9.675396873274071,4.)(-9.192114958332118,0.)
\psline[linewidth=1.6pt,linestyle=dashed,dash=3pt 3pt,linecolor=uququq](-9.192114958332118,0.)(-8.6878797775337,4.)
\psline[linewidth=1.6pt,linestyle=dashed,dash=3pt 3pt,linecolor=uququq](-7.7003626817933295,4.)(-9.192114958332118,0.)
\rput{0.}(-9.13125480011103,4.){\psellipse[linewidth=0.8pt,linecolor=uququq](0,0)(2.0415756510719025,0.6104851760967689)}
\rput{0.}(-11.704647063368565,4.){\psellipse[linewidth=0.8pt,linecolor=uququq](0,0)(1.4610297619681982,0.5255314159817697)}
\rput[tl](-12.97184784496106,5.8450681841831305){$N_B(x_i)$}
\rput[tl](-9.325219890625243,5.981136391434466){$N_B(x_{i+1})$}
\psline[linewidth=0.8pt,linestyle=dashed,dash=3pt 3pt,linecolor=uququq](-13.6,3.2)(-6.7,3.2)
\psline[linewidth=0.8pt,linestyle=dashed,dash=3pt 3pt,linecolor=uququq](-6.7,3.2)(-6.7,4.9)
\psline[linewidth=0.8pt,linestyle=dashed,dash=3pt 3pt,linecolor=uququq](-6.7,4.9)(-13.6,4.8)
\psline[linewidth=0.8pt,linestyle=dashed,dash=3pt 3pt,linecolor=uququq](-13.6,4.8)(-13.6,3.2)
\rput[tl](-6.467787538347179,4.620454318921101){$Y\cup Z$}
\psline[linewidth=0.8pt,linestyle=dashed,dash=3pt 3pt,linecolor=uququq](-12.8,-0.6)(-8.1,-0.6)
\psline[linewidth=0.8pt,linestyle=dashed,dash=3pt 3pt,linecolor=uququq](-8.1,-0.6)(-8.1,0.7)
\psline[linewidth=0.8pt,linestyle=dashed,dash=3pt 3pt,linecolor=uququq](-8.1,0.7)(-12.8,0.7)
\psline[linewidth=0.8pt,linestyle=dashed,dash=3pt 3pt,linecolor=uququq](-12.8,0.7)(-12.8,-0.6)
\rput[tl](-7.828469610860543,0.6744763086323373){$X$}
\begin{scriptsize}
\psdots[dotsize=3pt 0,dotstyle=*,linecolor=uququq](-2.,4.)
\psdots[dotsize=3pt 0,dotstyle=*,linecolor=uququq](10.,4.)
\psdots[dotsize=3pt 0,dotstyle=*,linecolor=uququq](4.,4.)
\psdots[dotsize=3pt 0,dotstyle=*,linecolor=uququq](7.,4.)
\psdots[dotsize=3pt 0,dotstyle=*,linecolor=uququq](1.,4.)
\psdots[dotsize=3pt 0,dotstyle=*,linecolor=uququq](13.,4.)
\psdots[dotsize=3pt 0,dotstyle=*,linecolor=uququq](7.743362938564151,0.)
\psdots[dotsize=3pt 0,dotstyle=*,linecolor=uququq](7.909204383634306,0.31055636004788233)
\psdots[dotsize=3pt 0,dotstyle=*,linecolor=uququq](-1.0370687869746776,0.3874887971621966)
\psdots[dotsize=3pt 0,dotstyle=*,linecolor=uququq](-1.3189702925479404,-0.054681504656606)
\psdots[dotsize=3pt 0,dotstyle=*,linecolor=uququq](-11.645987966519298,0.)
\rput[bl](-12.128224960002774,-0.4957102737291579){\uququq{$x_i$}}
\psdots[dotsize=3pt 0,dotstyle=*,linecolor=uququq](-9.192114958332118,0.)
\rput[bl](-9.678997229478718,-0.6317784809804945){\uququq{$x_{i+1}$}}
\psdots[dotsize=3pt 0,dotstyle=*,linecolor=uququq](-12.637948160495183,4.)
\psdots[dotsize=3pt 0,dotstyle=*,linecolor=uququq](-8.6878797775337,4.)
\psdots[dotsize=3pt 0,dotstyle=*,linecolor=uququq](-10.662913969014442,4.)
\psdots[dotsize=3pt 0,dotstyle=*,linecolor=uququq](-11.650431064754812,4.)
\psdots[dotsize=3pt 0,dotstyle=*,linecolor=uququq](-9.675396873274071,4.)
\psdots[dotsize=3pt 0,dotstyle=*,linecolor=uququq](-7.7003626817933295,4.)
\end{scriptsize}

\end{pspicture*}